\documentclass[11pt,a4paper]{article}

\usepackage[english]{babel}
\usepackage[utf8x]{inputenc}
\usepackage[T1]{fontenc}
\usepackage[a4paper,top=2cm,bottom=2cm,left=1.5cm,right=1.5cm,marginparwidth=1.75cm]{geometry}

\usepackage{amsmath,amsthm,amsfonts}
\usepackage{amssymb}
\usepackage{mathtools}
\usepackage{graphicx}
\usepackage{stmaryrd}
\usepackage[colorinlistoftodos]{todonotes}
\usepackage[colorlinks=true, allcolors=blue]{hyperref}
\usepackage{scalerel}
\usepackage{mathrsfs}
\usepackage{relsize}
\usepackage{enumerate}   
\usepackage{bm} 
\usepackage{dsfont}

\newtheorem{theorem}{Theorem}[section]
\newtheorem{lemma}[theorem]{Lemma}
\newtheorem{proposition}[theorem]{Proposition}
\newtheorem{corollary}[theorem]{Corollary}
\theoremstyle{definition}
\newtheorem{definition}[theorem]{Definition} 

\newtheorem{remark}[theorem]{Remark}

\theoremstyle{remark}
\newcommand{\C}{\mathbb{C}}

\newcommand{\R}{\mathbb{R}}
\newcommand{\N}{\mathbb{N}}
\newcommand{\Z}{\mathbb{Z}}
\newcommand{\SO}{\textnormal{\textbf{S}}_{0}}

\newcommand{\ghat}{\widehat{G}}

\newcommand{\s}{\textnormal{s}}
\newcommand{\NumWin}{n}
\newcommand{\NumDim}{d}
\newcommand{\La}{\Lambda}
\newcommand{\la}{\lambda}
\newcommand{\Lac}{\Lambda^\circ}
\newcommand{\lac}{\lambda^\circ}

\newcommand{\hhat}{\widehat{H}}
\newcommand{\restrict}[1]{\mathcal{R}_{#1}}
\newcommand{\period}[1]{\mathcal{P}_{#1}}
\newcommand{\cocy}{\textsf{c}}

\newcommand{\tr}{\textnormal{tr}}

\newcommand{\lmodule}{\mathcal{A}}
\newcommand{\rmodule}{\mathcal{B}}

\DeclareBoldMathCommand\boldlangle{\langle} 
\DeclareBoldMathCommand\boldrangle{\rangle}

\newcommand{\lhs}[2]{\prescript{}{\lmodule}{\boldlangle #1,#2\boldrangle}}
\newcommand{\rhs}[2]{{\boldlangle #1,#2\boldrangle}_{\!\rmodule}}
\newcommand{\glhs}[3]{\prescript{}{#3}{\boldlangle #1,#2\boldrangle}}
\newcommand{\grhs}[3]{{\boldlangle #1,#2\boldrangle}_{\!#3}}

\newcommand{\gmlhs}[3]{\prescript{}{#3}{{\textnormal{\textbf{[}}} #1,#2 \textnormal{\textbf{]}}}}
\newcommand{\gmrhs}[3]{ \textnormal{\textbf{[}} #1, #2 {\textnormal{\textbf{]}}}_{#3}}

\newcommand{\mvfun}[3]{
\ifthenelse{\equal{#2}{}}
	{\ifthenelse{\equal{#3}{}}  
    	{#1_{\bullet,\bullet}}{  
    	{#1_{\bullet,#3}}
    }}{ 
    {\ifthenelse{\equal{#3}{}}  
    	{#1_{#2,\bullet}}
    {  
    	{#1_{#2,#3}}}}}
} 

\newcommand{\vvfun}[2]{
\ifthenelse{\equal{#2}{}}
	{#1_{\bullet}}
  {#1_{#2}}
} 

\usepackage{fancyhdr}
\title{Sampling and periodization of generators of Heisenberg modules}
\author{Mads S.\ Jakobsen\thanks{Norwegian University of Science and Technology, Department of Mathematical Sciences, Trondheim, Norway, \mbox{E-mail: \protect\url{mads.jakobsen@ntnu.no}; \protect\url{franz.luef@ntnu.no}}}, Franz Luef\footnotemark[1]}

\begin{document}
\date{}
\maketitle

\begin{abstract}
This paper considers generators of Heisenberg modules in the case of twisted group $C^*$-algebras of closed subgroups of locally compact abelian groups and how the restrction and/or periodization of these generators yield generators for other Heisenberg modules. Since generators of Heisenberg modules are exactly the generators of (multi-window) Gabor frames, our methods are going to be from Gabor analy\-sis. In the latter setting the procedure of restriction and periodization of generators is well known. Our results extend this established part of Gabor analy\-sis  to the general setting of locally compact abelian groups. We give several concrete examples where we demonstrate some of the consequences of our results. Finally, we show that vector bundles over an irrational noncommutative torus may be approximated by vector bundles for finite-dimensional matrix algebras that converge to the irrational noncommutative torus with respect to the module norm of the generators, where the matrix algebras converge in the quantum Gromov-Hausdorff distance to the irrational noncommutative torus. 

\end{abstract}

\section{Introduction}

As shown in detail in \cite{jalu18} and \cite{lu11}, the construction of dual (multi-window) Gabor frame generators is equivalent to the construction of (matrix-valued) idempotent elements in twisted group $C^{*}$-algebras for closed subgroups of phase spaces represented by the Schr\"odinger representation of the Heisenberg group (as a special case we find the non-commutative tori generated by the translation and the modulation operator \cite{lu09,ri88}). Due to the mentioned equivalence, we present our results in a way that is understandable by members of both communities.

In the language of Gabor frames we generalize results on the sampling and periodization of dual Gabor frames generators developed in \cite{ja97-1,ka05,so07} from the Euclidean setting to the general setting of locally compact abelian (LCA) groups as well as to the multi-window case. For the Euclidean case our results widen the known theory, as the here developed results can also handle time-frequency shifts that come from general (not necessarily separable) subgroups of the time-frequency plane. The results given here also cover more abstract cases, e.g., sampling and periodization of Gabor frames for the square integrable functions over the adeles and over $\mathbb{Q}_{p}\times\R$ as constructed in \cite{enjalu18}. Furthermore, the results here can be applied to sample and periodize super (also known as vector valued) Gabor frames as well.

Concerning the Heisenberg modules, the results presented here and even those of \cite{ja97-1,ka05,so07} are completely new and have not been observed in the setting of the non-commutative geometry before. 
To give a good picture of what the results are, let us state a particular version of the known theory from \cite{so07} for the non-commutative torus $\mathcal{A}_{\theta}$.
We assume that $\theta$ is such that $\theta = a/M=b/N$ for some $a,b,M,N\in\N$ and take $d=aN$. The pre-$C^{*}$-algebra to this  $\mathcal{A}_{\theta}$ we realize as three different Banach algebras of operators. Specifically, we realize them as samples of the Schr\"odinger representations of the Heisenberg group that act on $L^{2}(\R)$, $\ell^{2}(a^{-1} \Z)$ and $\ell^{2}(\Z_{d})$, where $\Z_{d}=\Z/d \Z\cong\{0,1,\ldots,d-1\}$. Clearly, $\ell^{2}(\Z_{d}) \cong \C^{d}$. For convenience, we denote the algebras by $\lmodule_{\theta}^{\R}$, $\lmodule^{a^{-1} \Z}_{\theta}$ and $\lmodule^{\Z_d}_{\theta}$.
They are generated by the following unitary operators respectively,
\begin{align*} &U_{\R}f(t)  
= e^{2\pi i \theta t} f(t), \ \ V_{\R}f(t) 
= f(t-1), \ \ f\in L^{2}(\R), \ t\in\R,\\ 
&U_{a^{-1} \Z}f(t) 
= e^{2\pi i \theta t} f(t), \ \ V_{a^{-1}\Z}f(t) 
= f(t-1), \ \ f\in \ell^{2}(a^{-1}\Z), \ t\in a^{-1}\Z,\\ 
& U_{\Z_{d}} f(t) = e^{2\pi i b t/d} f(t), \ \ V_{\Z_{d}}f(t)= f(t-a), \ \ f\in\ell^{2}(\Z/d\Z), \ t\in \{0,1,\ldots,d-1\}.
\end{align*}
Observe that 
\[ U_{\R}V_{\R} = e^{2\pi i \theta} V_{\R} U_{\R}, \ \ U_{a^{-1}\Z}V_{a^{-1}\Z} = e^{2\pi i \theta} V_{a^{-1}\Z} U_{a^{-1}\Z}, \ \ U_{\Z_{d}}V_{\Z_d} = e^{2\pi i \theta} V_{\Z_d} U_{\Z_d}.\]
So $\lmodule_{\theta}^{\R}$, $\lmodule_{\theta}^{\R}$ and $\lmodule_{\theta}^{\R}$ are realizations of the non-commutative torus with parameter $\theta$.
To be consistent with later notation, we shall not so much use the operators $U$ and $V$ but rather the time-frequency shift operator $\pi$, defined as follows: 
\begin{enumerate}
    \item[(i)] For $(x,\omega)\in \R^{2}$ and $f\in L^{2}(\R)$ we define $\pi(x,\omega) f(t) = e^{2\pi i \omega t} f(t-x)$, $t\in\R$. Note that  $\pi(1,\theta)=U_{\R}V_{\R}$.
    \item [(ii)] For $(x,\omega)\in a^{-1}\Z\times[0,a)$ and $f\in \ell^{2}(a\Z)$ we define $\pi(x,\omega) f(t) = e^{2\pi i \omega t} f(t-x)$, $t\in a^{-1}\Z$. Note that $\pi(1,\theta) = U_{a^{-1}\Z} V_{a^{-1}\Z}$.
    \item[(iii)] For $(x,\omega)\in\Z_{d}\times\Z_{d}$ and $f\in \ell^{2}(\Z_{d})$ we define $\pi(x,\omega) f(t) = e^{2\pi i \omega t / d} f(t-x)$. Note that $\pi(a,b)=U_{\Z_{d}}V_{\Z_{d}}$. 
    
\end{enumerate}
From the time-frequency shifts we construct the following spaces,
\begin{align*}
    & \lmodule^{\R}_{\theta} = \big\{ {\bf a} \in \mathsf{B}(L^{2}(\R)) \, : \, {\bf{a}} =  \sum_{n,m\in\Z} \mathsf{a}(n,m) \, \pi(n,\theta m) \, , \ \mathsf{a}\in \ell^{1}(\Z^{2})\big\}, \\
    & \lmodule^{a^{-1}\Z}_{\theta} = \big\{ {\bf a} \in \mathsf{B}(\ell^{2}(a^{-1}\Z)) \, : \, {\bf{a}} =  \sum_{n\in\Z} \sum_{m=0}^{M-1} \mathsf{a}(n,m) \, \pi(n,\theta m) \, , \ \mathsf{a}\in \ell^{1}(\Z\times\Z_{M})\big\}, \\
    & \lmodule^{\Z_d}_{\theta} = \big\{ {\bf a} \in \mathsf{B}(\ell^{2}(\Z_{d})) \, : \, {\bf{a}} =   \sum_{n=0}^{N-1}\sum_{m=0}^{M-1}  \mathsf{a}(n,m) \, \pi(na,mb) \, , \ \mathsf{a}\in \ell^{1}(\Z_{N}\times\Z_{M})\big\}.
\end{align*}
The norm $\Vert \mathbf{a} \Vert = \Vert \mathsf{a} \Vert_{1}$ turns each of them into a Banach algebra with respect to composition of operators and the taking of $L^{2}$-adjoints.

For functions in \emph{Feichtinger's algebra} $\SO(\R)$ (see Section \ref{sec:fei-alg}), sequences in $\ell^{1}(a^{-1}\Z)$, and vectors in $\C^{d}$ we define a respective $\lmodule_{\theta}$-valued inner-product in the following way:
\begin{align*} \glhs{\,\cdot\,}{\,\cdot\,}{\R} & : \SO(\R)\times \SO(\R) \to \lmodule^{\R}_{\theta}, \ \glhs{f}{g}{\R} = \sum_{n,m\in\Z} \langle f, \pi(n,\theta m) g \rangle \, \pi(n,\theta m), \\
\glhs{\,\cdot\,}{\,\cdot\,}{a^{-1}\Z} & : \ell^{1}(a^{-1} \Z)\times \ell^{1}(a^{-1}\Z) \to \lmodule^{a^{-1}\Z}_{\theta}, \ \glhs{f}{g}{a^{-1}\Z} = \sum_{n\in\Z}\sum_{m=0}^{M-1} \langle f, \pi(n,\theta m) g \rangle \, \pi(n,\theta m) ,\\
\glhs{\,\cdot\,}{\,\cdot\,}{\Z_d} & : \C^{d}\times \C^{d} \to \lmodule^{\Z_d}_{\theta}, \ \glhs{f}{g}{\Z_d} = \sum_{n=0}^{N-1}\sum_{m=0}^{M-1} \langle f, \pi(na,mb) g \rangle \, \pi(na,mb),
\end{align*}
The un-annotated inner products are the usual ones on the Hilbert space $L^{2}$: for $f,g\in L^{2}(\R)$ (and particular for functions in $\SO(\R)$) we have $\langle f,g\rangle = \int_{R} f(t) \, \overline{g(t)} \, dt$, where $dt$ is the Lebesgue measure. For $f,g\in \ell^{2}(a^{-1}\Z)$ (and particular for sequences in $\ell^{1}(a^{-1}\Z)$) we have $\langle f,g\rangle = \sum_{t\in a^{-1}\Z} f(t) \, \overline{g(t)}$. For $f,g\in\C^{d}$ we have $\langle f,g\rangle = \sum_{t=0}^{d-1} f(t) \, \overline{g(t)}$.

The \emph{module norm} of a function in $\SO(\R)$, a sequence in $\ell^{1}(a^{-1}\Z)$ and a vector in $\C^{d}$ is given, respectively,  by
\begin{align*}
   & \Vert g \Vert_{\lmodule^{\R}_{\theta}} = \big\Vert \glhs{g}{g}{{\R}} \big\Vert_{\textnormal{op},L^{2}}^{1/2},\ \ g\in\SO(\R), \\
   & \Vert g \Vert_{\lmodule^{a^{-1}\Z}_{\theta}} = \big\Vert \glhs{g}{g}{{a^{-1}\Z}} \big\Vert_{\textnormal{op},\ell^{2}}^{1/2}, \ \ g\in\ell^{1}(a^{-1}\Z),\\
   & \Vert g \Vert_{\lmodule^{\Z_d}_{\theta}} = \big\Vert \glhs{g}{g}{\Z_d} \big\Vert_{\textnormal{op},\ell^{2}(\Z_{d})}^{1/2}, \ \ g\in\C^{d}.
\end{align*}

Established results in the theory of Gabor frames, and especially concerning the sampling and periodization of Gabor frame generators \cite{so07}, directly translate into the following statements.
\begin{theorem} \label{th:intro}
Let all notation be as above. If $g$ is a function in $\SO(\R)$ (or particularly in the Schwartz space) such that $\glhs{g}{g}{\R}$ is a projection in $\lmodule^{\R}_{\theta}$, then the following holds.
\begin{enumerate}
    \item[(i)] The module norm of $g$ satisfies
    \[ \Vert g \Vert_{\lmodule^{\R}_{\theta}} \le C := \theta^{-1} \sum_{m,n\in\Z} \vert \langle g, e^{2\pi i m (\cdot)} g( \cdot - n\theta^{-1})\rangle \vert.\]
    \item[(ii)] The sequence $\tilde{g}:= \{ \sqrt{a^{-1}} \, g(t)\}_{t\in a^{-1}\Z}$ belongs to $\ell^{1}(a^{-1}\Z)$ and is such that $\glhs{\tilde{g}}{\tilde{g}}{a^{-1}\Z}$ is a projection in $\lmodule^{a^{-1}\Z}_{\theta}$. Moreover, the module norm of $\tilde{g}$ satisfies $\Vert \tilde{g} \Vert_{\lmodule^{a^{-1}\Z}_{\theta}} \le C$.
    \item[(iii)] The finite sequence $\tilde{\tilde{g}}(t) :=  \sqrt{a^{-1}} \sum_{k\in\Z} g( a^{-1}(t-kd))$, $t\in\{0,1,\ldots, d-1\}$ belongs to $\C^{d}$ and is such that $\glhs{\tilde{\tilde{g}}}{\tilde{\tilde{g}}}{\Z_{d}}$ is a projection in $\lmodule^{\Z_{d}}_{\theta}$. Moreover, the module norm of $\tilde{\tilde{g}}$ satisfies $\Vert \tilde{\tilde{g}} \Vert_{\lmodule^{\Z_d}_{\theta}} \le C$. 
\end{enumerate}
\end{theorem}

The purpose of this note is to generalize Theorem \ref{th:intro} to the setting of functions over locally compact abelian groups and the  associated Heisenberg modules and Gabor systems as described in \cite{jalu18}. We do this in Section \ref{sec:sampling} and \ref{sec:periodization}. \\
Specifically, our main results are the sampling and periodization theorem for the generators of \emph{matrix valued projections in Banach algebras} and, equivalently, for \emph{dual multi-window Gabor frames} that are generated by the time-frequency shifts from closed subgroups of the time-frequency plane of general locally compact abelian groups in Theorem \ref{th:sampl-non-sep} and Theorem \ref{th:period-non-sep}.

Due to the abstract nature of these results we give a number of concrete examples in Section \ref{sec:examples}. 

In Section \ref{sec:prelim} we state some necessary terminology on Fourier analysis on groups, on the Feichtinger algebra (Section \ref{sec:fei-alg}) and the connection between multi-window Gabor frames and matrix-valued projections in Heisenberg modules (Section \ref{sec:frames-and-modules}).

In Section \ref{sec:from-discrete-to-continuous} we state results that are of independent interest in the matter of projections for the tori described here in the introduction. Observe that Theorem \ref{th:intro} only applies to non-commutative tori where $\theta$ is rational. In Section \ref{sec:from-discrete-to-continuous} we translate known results in Gabor analysis into the language of NC-tori to give an approach for the irrational case. Furthermore, we state results that make it possible to construct generators of projections in $\lmodule_{\theta}^{\R}$ from the sequences $\tilde{g}$ and $\tilde{\tilde{g}}$ obtained via Theorem \ref{th:intro}. The results in Section \ref{sec:from-discrete-to-continuous} are based on the theory of Gabor frames established in \cite{feka04,feka07,grle04,ka05}. These results indicate that a natural measure for projective finitely generated modules in terms of the difference of the generators in the module norm and hence two such modules are close if their generators are in the module norm. Our results show that this is the case for Heisenberg modules over $\lmodule_{\theta}^{\R}$. If $\lmodule_{\theta}^{\R}$ is the irrational noncommutative torus, then we show that for a sequence of matrix algebras converging to $\lmodule_{\theta}^{\R}$ in the sense of Rieffel's quantum Gromov-Hausdorff distance, then one can use the generators of Heisenberg modules over these matrix algebras can be turned into generators of $\lmodule_{\theta}^{\R}$ and that these generators converge with respect to the module norm.

\section{Preliminaries}
\label{sec:prelim}
We let $G$ be a locally compact Hausdorff abelian topological (LCA) group and let $\ghat$ be its dual group. 
The action of a character $\omega\in \ghat$ on an element $x\in G$ is written as $\omega(x)$. We assume some fixed Haar measure $\mu_{G}$ on $G$ and we normalize the Haar measure $\mu_{\ghat}$ on $\ghat$ in the unique way such that the Fourier inversion holds. That is, if $f\in L^{1}(G)$ is such that its Fourier transform, $\mathcal{F}f(\omega) = \hat{f}(\omega) = \int_{G} f(t) \, \overline{\omega(t)} \, dt$, $\omega\in \ghat$ is a function in $L^{1}(\ghat)$, then
\[ f(t) = \int_{\ghat} \hat{f}(\omega) \, \omega(t) \, d\omega \ \ \text{for all} \ \ t\in G.\]
We equip $L^{2}(G)$ with the inner product $\langle f, g\rangle = \int_{G} f(t) \overline{g(t)} \, dt$ which is linear in the first entry. The Fourier transform extends to a unitary operator on $L^{2}(G)$.

For any $x\in G$ and $\omega\in\ghat$ we define the translation operator (time-shift) $T_{x}$ and the modulation operator (frequency-shift) $E_{\omega}$ by
\[ T_{x}f(t) = f(t-x) \ \ \text{and} \ \ E_{\omega} f(t) = \omega(t) f(t), \ \ t\in G,\]
where $f$ is a complex-valued function on $G$.
Observe that 
\[ \mathcal{F}T_{x} = E_{-\omega} \mathcal{F} \ \ , \ \ \mathcal{F}E_{\omega}=T_{\omega} \mathcal{F} \ \ , \ \  E_{\omega}T_{x} = \omega(x) T_{x} E_{\omega} .\]
For any $\chi = (x,\omega)\in G\times\ghat$
we define the \emph{time-frequency shift operator} 
\[ \pi(\chi) \equiv \pi(x,\omega) := E_{\omega} T_{x}.\]
It is clear that time-frequency shift operators are unitary on $L^{2}(G)$. 

For two elements $\chi_{1}=(x_{1},\omega_{1})$ and $\chi_{2}=(x_{2},\omega_{2})$ in $G\times\ghat$ we define the \emph{cocycle}
\begin{equation} \label{eq:def-cocy} \cocy : (G\times\ghat)\times (G\times\ghat) \to \mathbb{T}, \ \cocy(\chi_{1},\chi_{2}) = \overline{\omega_{2}(x_{1})} \end{equation}
and the associated \emph{symplectic cocyle}
\begin{equation} \label{eq:def-cocy-sym} \cocy_s : (G\times\ghat)\times (G\times\ghat) \to \mathbb{T}, \ \cocy_s(\chi_{1},\chi_{2}) = \cocy(\chi_{1},\chi_{2}) \, \overline{\cocy(\chi_{2},\chi_{1})} = \overline{\omega_{2}(x_{1})} \, \omega_{1}(x_{2}). \end{equation}
For any $\chi,\chi_{1},\chi_{2},\chi_{3}\in G\times\ghat$ the cocycle and time-frequency shift satisfy the following,
\begin{align*}
\overline{\cocy(\chi_{1},\chi_{2})} & = \cocy(-\chi_{1},\chi_{2}) = \cocy(\chi_{1},-\chi_{2}), \\
\cocy(\chi_{1}+\chi_{2},\chi_{3}) & = \cocy(\chi_{1},\chi_{3}) \, \cocy(\chi_{2},\chi_{3}), \ \ \cocy(\chi_{1},\chi_{2}+\chi_{3}) = \cocy(\chi_{1},\chi_{2}) \, \cocy(\chi_{1},\chi_{3}), \\
 \pi(\chi_{1}) \, \pi(\chi_{2}) &  = \cocy(\chi_{1},\chi_{2}) \, \pi(\chi_{1}+\chi_{2}), \\
 \pi(\chi_{1}) \, \pi(\chi_{2}) & = \cocy_s(\chi_{1},\chi_{2}) \,\pi(\chi_{2}) \, \pi(\chi_{1}), \\
 \pi(\chi)^* & = \cocy(\chi,\chi) \, \pi(-\chi), \\
 \pi(\chi_{1})^{*} \, \pi(\chi_{2})^{*} & = \overline{\cocy(\chi_{2},\chi_{1})} \, \pi(\chi_{1}+\chi_{2})^{*}.
\end{align*}
The \emph{short-time Fourier transform} with respect to a given function $g\in L^{2}(G)$ is the operator
\begin{equation} \label{eq:def-STFT} \mathcal{V}_{g} : L^{2}(G)\to L^{2}(G\times\ghat), \ \mathcal{V}_{g}{f}(\chi) = \langle f, \pi(\chi) g\rangle, \ \chi\in G\times\ghat. \end{equation}
The operator $\mathcal{V}_{g}^{*} \circ \mathcal{V}_{g}$ is a multiple of the identity. Specifically, for all $f_{1},f_{2},g,h\in L^{2}(G)$
\begin{align} \label{eq:STFT} \langle f_{1},f_{2} \rangle \, \langle h, g\rangle & = \langle \mathcal{V}_{g}f_{1}, \mathcal{V}_{h}f_{2}\rangle \\
& = \int_{G\times\ghat} \langle f, \pi(\chi) g \rangle \, \langle \pi(\chi) h , f_{2}\rangle \, d\mu_{G\times\ghat}(\chi). \nonumber \end{align}

The symbol $\Lambda$ will always denote a closed subgroup of the time-frequency plane $G\times\ghat$. The induced topology and group action on $\Lambda$ and on the quotient group $(G\times\ghat)/\Lambda$ turn those into LCA groups as well, and can therefore be equipped with their own Haar measures. If the measures on $G$, $\ghat$ and $\Lambda$ are fixed, then the Haar measure $\mu_{(G\times\ghat)/\Lambda}$ on the quotient group $(G\times\ghat)/\Lambda$ can be uniquely scaled such that, for all $f\in L^{1}(G\times\ghat)$,
\begin{equation} \label{eq:2503a} \int_{G\times\ghat} f(\chi) \, d\mu_{G\times\ghat}(\chi) = \int_{(G\times\ghat)/\Lambda} \int_{\Lambda} f(\chi+\lambda) \, d\mu_{\Lambda}(\lambda) \, d\mu_{(G\times\ghat)/\Lambda}(\dot{\chi}) \ \ \dot{\chi} = \chi + \Lambda, \ \chi\in G\times\ghat.\end{equation}
If \eqref{eq:2503a} holds we say that $\mu_{G\times\ghat}$, $\mu_{\Lambda}$ and $\mu_{(G\times\ghat)/\Lambda}$ are \emph{canonically related} and the equality in \eqref{eq:2503a} is called \emph{Weil's formula}. We always choose the measures $\mu_{G\times\ghat}$, $\mu_{\Lambda}$ and $\mu_{(G\times\ghat)/\Lambda}$ in this way. For more on this, see \cite[p.87-88]{rest00} and \cite[Theorem~3.4.6]{rest00}.
With the uniquely determined measure $\mu_{(G\times\ghat)/\Lambda}$ we define the \emph{size} or the \emph{covolume} of $\Lambda$, by $\s(\Lambda)=\int_{(G\times\ghat)/\Lambda} 1 \, d\mu_{(G\times\ghat)/\Lambda}$. Note that $\s(\Lambda)$ is finite if and only if $\Lambda$ is a \emph{co-compact} subgroup of $G\times\ghat$, i.e., the quotient group $(G\times\ghat)/\Lambda$ is compact. If $\Lambda$ is discrete, co-compact (hence a lattice), and equipped with the counting measure, then $\s(\Lambda)$ is exactly the measure of any of its fundamental domains. The \emph{adjoint} group of $\Lambda$ is the closed subgroup of $G\times\ghat$ given by
\begin{align*} \Lambda^{\circ} 
& = \{ \chi \in G\times\ghat \, : \, \cocy_{s}(\chi,\lambda) = 1 \ \ \textnormal{for all} \ \ \lambda\in\Lambda  \ \}. \end{align*}

For any closed subgroup $\Lambda$ one has $(\Lambda^{\circ})^{\circ}=\Lambda$ and $\widehat{\Lambda^{\circ}}\cong (G\times\ghat)/\Lambda$. Given these identifications, we take the Haar measure $\mu_{\Lambda^{\circ}}$ on $\Lambda^{\circ}$ such that the Fourier inversion between functions on $\Lambda^{\circ}$ and $(G\times\ghat)/\Lambda$ holds. This unique measure on $\Lambda^{\circ}$ is called the \emph{orthogonal measure} relative to $\mu_{\Lambda}$ \cite[Definition 5.5.1]{rest00}. We now choose the Haar measure on $(G\times\ghat)/\Lambda^{\circ}$ such that the measures $\mu_{(G\times\ghat)}$, $\mu_{\Lambda^{\circ}}$ and $\mu_{(G\times\ghat)/\Lambda^{\circ}}$ are canonically related. This ensures that also the Fourier inversion formula between functions on $\Lambda$ and $(G\times\ghat)/\Lambda^{\circ}$ holds \cite[Theorem 5.5.12]{rest00}. 
\begin{remark} \label{rem:orth-measure-for-cocompact}
For a closed subgroup $\Lambda$ with measure $\mu_{\Lambda}$ it is in general difficult to say more about the orthogonal measure on $\mu_{\Lambda^{\circ}}$ on $\Lambda^{\circ}$. However, if the quotient group $(G\times\ghat)/\Lambda$ is compact (equivalently $\Lambda^{\circ}$ is discrete), then the orthogonal measure on $\Lambda^{\circ}$ satisfies
\begin{equation} \label{eq:cocompact-orth-measure} \int_{\Lambda^{\circ}} f(\lambda^{\circ}) \, d\mu_{\Lambda^{\circ}}(\lambda^{\circ}) = \frac{1}{\s(\Lambda)} \sum_{\lambda^{\circ}\in\Lambda^{\circ}} f(\lambda^{\circ}) \ \ \textnormal{for all} \ \ f\in \ell^{1}(\Lambda^{\circ}).\end{equation}
\end{remark}

For more on harmonic analysis on locally compact abelian groups see the book by Reiter and Stegeman \cite{rest00}. Other books are the one by Folland \cite{fo16} and Hewitt and Ross \cite{hero70,hero79}.

\subsection{The Feichtinger algebra}
\label{sec:fei-alg}
For any LCA group $G$ the \emph{Feichtinger algebra}  $\SO(G)$ \cite{fe81-2,ja19,lo80} (sometimes denoted by $\mathbf{M}^{1}(G)$) is the set of functions given by 
%
\[ \SO(G) = \big\{ f\in L^{2}(G) \, : \, \mathcal{V}_{f}f \in L^{1}(G\times\ghat) \big\} 
.\]
For the definition of $\mathcal{V}_{f}f$ see \eqref{eq:def-STFT}. Any non-zero function $g\in \SO(G)$ can be used to define a norm on $\SO(G)$, 
\begin{equation} \label{eq:norm} \Vert \cdot \Vert_{\SO(G),g} : \SO(G) \to \R^{+}_{0}, \ \Vert f \Vert_{\SO(G),g} = \Vert \mathcal{V}_{g}f \Vert_{1}. \end{equation}
All norms defined in this way are equivalent \cite[Proposition 4.10]{ja19} and they turn $\SO(G)$ into a Banach space \cite[Theorem 4.12]{ja19}. The usefulness of the Feichtinger algebra $\SO(G)$ lies in the fact that it behaves very much like the Schwartz-Bruhat space $\mathscr{S}(G)$ (also, one has the inclusion $\mathscr{S}(G)\subset \SO(G)$, see \cite[Theorem 9]{fe81-2}). 

The construction of projective modules over the twisted group algebra in Rieffel's setting \cite{ri88} requires one to have a function space that allows us to define actions and innerproducts with values in $L^1$, and Feichtinger's algebra turns out to be the most convenient choice. In some problems it has also turned out to be of relevance that we are in the position to deal with settings beyond the smooth one \cite{dajalalu18,lu09,lu11,jalu18}. 

Among its properties, we mention the following ones. Properties (vi)-(ix) are of special importance to us here.

\begin{lemma}\label{le:s0-properties} \label{le:SO-properties} 
\begin{enumerate}[(i)]
\item All functions in $\SO(G)$ are continuous, absolutely integrable, and vanish at infinity.
\item If $G$ is discrete, then $(\SO(G), \Vert \cdot \Vert_{\SO})=(\ell^{1}(G), \Vert \cdot \Vert_{1})$.
\item Time-frequency shifts $\pi(\chi)$, $\chi\in G\times\ghat$, are an isometry on $\SO(G)$. The Fourier transform is a continuous bijection from $\SO(G)$ onto $\SO(\ghat)$. 
\item $\SO(G)$ is continuously embedded into $L^{p}(G)$ for all $p\in[1,\infty]$. 
In fact, if $1/p+1/q=1$, then
\[ \Vert f \Vert_{p} \le \Vert g \Vert_{q}^{-1} \Vert f \Vert_{\SO,g} \ \ \text{for all} \ \ f\in \SO(G).\]
\item $\SO(G)$ is a Banach algebra with respect to convolution and point-wise multiplication. 
\item For any closed subgroup $H$ of $G$, the restriction operator 
\[ \mathcal{R}_{H} : \SO(G) \to \SO(H), \ \mathcal{R}_{H}f(t) = f(t), \ \ t\in H,\]
is a linear, bounded and surjective operator from $\SO(G)$ onto $\SO(H)$.
\item For any closed subgroup $H$ of $G$, the periodization operator 
\[ \mathcal{P}_{H} : \SO(G) \to \SO(G/H), \ \mathcal{P}_{H} f(\dot t) = \int_{H} f(t+x) \, d\mu_{H}(x), \ \ \dot t = t+H, \ t\in H\]
is a linear, bounded and surjective operator from $\SO(G)$ onto $\SO(G/H)$.
\item For any $f,g\in\SO(G)$ the short-time Fourier transform $\mathcal{V}_{g}{f}$ is a function in $\SO(G\times\ghat)$. Also, there exists is a constant $c>0$ such that $\Vert\mathcal{V}_{g}f\Vert_{\SO} \le c \, \Vert f \Vert_{\SO} \, \Vert g \Vert_{\SO}$ for all $f,g\in \SO(G)$. 
\item The Poisson (summation) formula holds pointwise for all functions in $\SO(G)$. That is, for any closed subgroup $H$ of $G$ and any $f\in \SO(G)$
\[ \int_{H} f(t) \, d\mu_{H}(t) = \int_{H^{\perp}} \hat{f}(\gamma) \, d\mu_{H^{\perp}}.\]
If $H$ is a closed co-compact subgroup of $G$, then the Poisson formula takes the form
\begin{equation} \label{eq:poisson-for-cocompact}  \int_{H} f(t) \, d\mu_{H}(t) = \frac{1}{\textnormal{s}(H)} \sum_{\gamma\in H^{\perp}} \hat{f}(\gamma). \end{equation}
\end{enumerate}
\end{lemma}
\begin{proof}(i). This follows from \cite[Definition 1]{fe81-2} or \cite[Lemma 4.19]{ja19}. (ii). see 
\cite[Remark 3]{fe81-2} or \cite[Lemma 4.11]{ja19}. (iii). \cite[eq.\ (4.12), Example 5.2(i,iii), ]{ja19}. (iv). That $\SO$ is continuously embedded into $L^{p}$ follows from the fact that that $\SO(G) = W(\mathcal{F}L^{1},L^{1})$ (\cite[Remark 6]{fe81-2}) together with the inclusions in \cite[Lemma 1.2(iv)]{fe81-1} and the fact that $W(L^{p},L^{p}) = L^{p}$ \cite[Lemma 1.2(i)]{fe81-1}. For the inequality see \cite[Lemma 4.11]{ja19}.  (v). $\SO$ is a Segal algebra (\cite[Theorem 1]{fe81-2}) and any Segal algebra is a convolution algebra \cite[\S 4]{re71}. By (iii) this implies that it is also an algebra under pointwise multiplication. Alternatively, see \cite[Corollary 4.14]{ja19}. (vi+vii). See \cite[Theorem 7]{fe81-2} or \cite[Theorem 5.7]{ja19}. (viii). \cite[Theorem 5.3(ii)]{ja19}. (ix). That the Poisson formula holds for functions in $\SO$ is stated in \cite[Remark 15]{fe81-1}. Alternatively, see \cite[Theorem 5.7, Example 5.11]{ja19}. 
\end{proof}

\subsection{Gabor frames and Heisenberg modules}
\label{sec:frames-and-modules}
The following is a summary of certain results and facts that can, unless specififed otherwise, be found in \cite{jalu18}. Let $\Lambda$ be a closed subgroup of the time-frequency plane $G \times \widehat{G}$ and let $\Lambda^{\circ}$ be its adjoint group. 
We use the integrated Schr\"odinger representation to define the following two Banach algebras,
\begin{align*} 
\lmodule & = \big\{ {\bf a} \in \mathsf{B}(L^{2}(G)) \, : \, {\bf{a}} = \int_{\Lambda} \mathsf{a}(\lambda) \, \pi(\lambda) \, d\lambda, \ \mathsf{a}\in \SO(\Lambda)\big\}, \\
\rmodule & = \big\{ {\bf b} \in \mathsf{B}(L^{2}(G)) \, : \, {\bf{b}} = \int_{\Lambda^{\circ}} \mathsf{b}(\lambda^{\circ}) \, \pi(\lambda^{\circ})^{*} \, d\lambda^{\circ}, \ \mathsf{b}\in \SO(\Lambda^{\circ})\big\}. \end{align*}
Indeed, the norm $\Vert {\bf{a}} \Vert_{\lmodule} = \Vert  \mathsf{a} \Vert_{\SO}$ (where ${\bf a}$ and $\mathsf{a}$ are related as in the definition of $\lmodule$) turns $\lmodule$ into an involutive Banach algebra with respect to composition of operators and where the involution is the $L^{2}$-adjoint. Similarly, $\rmodule$ becomes an involutive Banach algebra.
\begin{remark} In the definition of $\rmodule$ the measure on $\Lambda^{\circ}$ is the measure that is orthogonal to the measure on $\Lambda$, cf.\ Remark \ref{rem:orth-measure-for-cocompact}. Hence, if $\Lambda$ is a co-compact subgroup of $G\times\ghat$ (e.g., a lattice), then
\[ \rmodule = \big\{ {\bf b} \in \mathsf{B}(L^{2}(G)) \, : \, {\bf{b}} = \frac{1}{s(\Lambda)} \sum_{\lambda^{\circ}\in\Lambda^{\circ}} \mathsf{b}(\lambda^{\circ}) \, \pi(\lambda^{\circ})^{*} \, d\lambda^{\circ}, \ \mathsf{b}\in \SO(\Lambda^{\circ})\big\}. \]
Since $\Lac$ is discrete we have $\SO(\Lac)=\ell^{1}(\Lac)$ (cf.\ Lemma \ref{le:SO-properties}(ii)), and so $\Vert { \bf b} \Vert_{\rmodule} = \Vert \mathsf{b} \Vert_{\SO} = \Vert \mathsf{b} \Vert_{1}$. 
\end{remark}
The \emph{traces} on both $\lmodule$ and $\rmodule$ are given by the continuous operators
\[ \tr_{\lmodule} : \lmodule \to \C,  \ \tr_{\lmodule} (\mathbf{a}) = \mathsf{a}(0) \ , \ \ \ \ \tr_{\rmodule} : \rmodule \to \C,  \ \tr_{\rmodule} (\mathbf{b}) = \mathsf{b}(0). \]
Elements of $\lmodule$ act from the left on functions in $L^{2}(G)$ by 
\[ \mathbf{a} \cdot f := \int_{\Lambda} \mathsf{a}(\lambda) \pi(\lambda) f \, d\lambda, \ \ f\in L^{2}(G), \ \mathbf{a}\in \lmodule.\]
Operators in $\rmodule$ act from the right on $L^{2}(G)$,
\[ f \cdot \mathbf{b} := \int_{\Lambda^{\circ}} \mathsf{b}(\lambda^\circ)\, \pi(\lambda^{\circ})^* f \, d\lambda^{\circ}, \ \ f\in L^{2}(G), \ \mathbf{b}\in \rmodule.\]
We define $\lmodule$- and $\rmodule$-valued inner products in the following way: 
\begin{align*} \glhs{\,\cdot\,}{\,\cdot\,}{\Lambda} & : \SO(G)\times \SO(G) \to \lmodule, \ \glhs{f}{g}{\Lambda} = \int_{\Lambda} \langle f, \pi(\lambda) g \rangle \, \pi(\lambda) \, d\lambda,\\
\grhs{\,\cdot\,}{\,\cdot\,}{\Lac} & :\SO(G)\times\SO(G)\to \rmodule, \ \grhs{f}{g}{\Lac} = \int_{\Lambda^{\circ}} \langle g, \pi(\lambda^{\circ})^{*}f\rangle \,\pi(\lambda^{\circ})^{*} \, d\lambda^{\circ}. \end{align*}
\begin{remark}
In \cite{jalu18} the notation $\lhs{\,\cdot\,}{\,\cdot\,}$ and $\rhs{\,\cdot\,}{\,\cdot\,}$ is used for the $\lmodule$- and $\rmodule$-valued inner products, respectively. However, for our purposes it will prove useful that the inner products reflect the subgroup of the time-frequency plane that is used, i.e., $\La$ and $\Lac$.
\end{remark}
The $\lmodule$- and $\rmodule$-valued inner products satisfy the associativity condition,
\begin{equation} \label{eq:2302c} \glhs{f}{g}{\La} \cdot h = f\cdot\grhs{g}{h}{\Lac} \ \ \text{for all} \ \ f,g,h\in\SO(G).\end{equation}
I.e., $\int_{\Lambda} \langle f, \pi(\la) g\rangle \, \pi(\la) h \, d\la = \int_{\Lac} \langle h, \pi(\lac)^{*} g\rangle \, \pi(\lac)^{*} f \, d\lac$ for any $f,g,h\in\SO(G)$.
In time-frequency analysis, this equality is known as the \emph{fundamental-identity of Gabor analysis}.
We define the $\lmodule$-module and $\rmodule$-module norm to be $\|g\|_\Lambda:=\|\glhs{g}{g}{\La}\|^{1/2}_{\textnormal{op},L^{2}}$ and $\|g\|_{\Lambda^\circ}:=\|\grhs{g}{g}{\Lac}\|^{1/2}_{\textnormal{op},L^{2}}$, respectively. It is a fact that $\Vert g \Vert_{\La} = \Vert g \Vert_{\Lac}$.
Observe that
\[ \tr_{\lmodule} \big( \lhs{f}{g}\big) = \tr_{\rmodule} \big(\rhs{g}{f}\big) = \langle f, g\rangle \ \ \text{for all} \ \ f,g\in\SO(G). \]
Rather than just the Banach algebras $\lmodule$ and $\rmodule$, we wish to consider matrices that consist of such elements. Thus, for $n\in\N$ we let $\textnormal{M}_{n}(\lmodule)$ be the set of all $\lmodule$-valued $n\times n$ matrices $(\mathbf{a}_{j,k})$, $\mathbf{a}_{j,k}\in \lmodule$, $j,k\in\Z_{n}$. Elements in $\textnormal{M}_{n}(\lmodule)$ have the natural left-action on $n$-tuple of functions in $L^{2}(G)$, $L^{2}(G)^{n}$.\footnote{Rather than $n$-tuples of functions in $L^{2}(G)$ one can, equivalently, think of functions in $L^{2}(G\times\Z_{n})$ or of vector valued functions $L^{2}(G;\C^{n})$.} It is given by matrix-vector multiplication,
\[ (\mathbf{a}_{j,k})_{j,k\in\Z_{n}} \cdot (f_{j})_{j\in\Z_{n}} = \big( \sum_{k\in\Z_{\NumWin}} \mathbf{a}_{j,k} \cdot f_{k} \big)_{j\in\Z_{n}}.\]
We define the $\textnormal{M}_{n}(\lmodule)$-valued inner product on $\SO(G)^{n}$ as follows:
\begin{align*} & \gmlhs{\,\cdot\,}{\,\cdot\,}{\La} : \SO(G)^{\NumWin}\times \SO(G)^{\NumWin} \to \textnormal{M}_{\NumDim\cdot\NumWin}(\lmodule), \ \gmlhs{(f_{j})}{(g_{j})}{\La} = \begin{bmatrix} 
\glhs{f_{1}}{g_{1}}{\La} & \glhs{f_{1}}{g_{2}}{\La} & \cdots & \glhs{f_{1}}{g_{n}}{\La} \\
\glhs{f_{2}}{g_{1}}{\La} & \glhs{f_{2}}{g_{2}}{\La} & \cdots & \glhs{f_{2}}{g_{n}}{\La} \\
\vdots & \vdots & \ddots & \vdots \\
\glhs{f_{\NumWin}}{g_{1}}{\La} & \glhs{f_{\NumWin}}{g_{2}}{\La} & \cdots & \glhs{f_{\NumWin}}{g_{n}}{\La} 
\end{bmatrix}.\end{align*}
We use the square brackets $\gmlhs{\,\cdot\,}{\,\cdot\,}{\La}$ to distinguish the $\textnormal{M}_{n}(\lmodule)$-valued inner product from the $\lmodule$-valued inner product $\glhs{\,\cdot\,}{\,\cdot\,}{\La}$. For $\NumWin=1$ these two notions coincide.

For $n\in\N$ we let $\textnormal{M}_{n}(\rmodule)$ be the set of all $\rmodule$-valued $n\times n$ matrices $(\mathbf{b}_{j,k})$, $\mathbf{b}_{j,k}\in\rmodule$, $j,k\in\Z_{n}$. These have a natural right-action on $L^{2}(G)^{n}$, given by vector-matrix multiplication,
\[ (f_{j})_{j\in\Z_{n}} \cdot (\mathbf{b}_{j,k})_{j,k\in\Z_{n}} = \big( \sum_{k\in\Z_{\NumWin}} f_{k} \cdot \mathbf{b}_{k,j} \big)_{j\in\Z_{n}}. \]
The $\textnormal{M}_{n}(\rmodule)$-valued inner product is 
\[ \gmrhs{\,\cdot\,}{\,\cdot\,}{\Lac} : \SO(G)^{n}\times \SO(G)^{n} \to \textnormal{M}_{n}(\rmodule), \ \gmrhs{(f_{j})}{(g_{j})}{\Lac} = \textnormal{diag} \big( 
\sum_{j\in\Z_{n}} \grhs{f_{j}}{g_{j}}{\Lac} \big).\]
Note that, in general, $\gmlhs{\,\cdot\,}{\,\cdot\,}{\La}$ is a full matrix, where as $\gmrhs{\,\cdot\,}{\,\cdot\,}{\Lac}$ is a diagonal matrix.
By use of \eqref{eq:2302c} it is immediate that the matrix valued inner-products satisfy Rieffel's associativity condition
\begin{equation} \label{eq:associativity-for-mws} \gmlhs{(f_{j})}{(g_{j})}{\La}\cdot (h_{j}) = (f_{j}) \cdot \gmrhs{(g_{j})}{(h_{j})}{\Lac} \ \ \text{for all} \ \ (f_{j}),(g_{j}),(h_{j})\in \SO(G)^{n}.\end{equation}
The trace on elements in $\textnormal{M}_{n}(\lmodule)$ and $\textnormal{M}_{n}(\rmodule)$ are given by 
\begin{align*} 
& \tr_{\textnormal{M}(\lmodule)} : \textnormal{M}_{n} (\lmodule) \to \C , \ \tr_{\textnormal{M}(\lmodule)}\big( (\mathbf{a}_{j,k}) \big) = \sum_{j\in\Z_{n}} \tr_{\lmodule} (\mathbf{a}_{j,j}), \ \ (\mathbf{a}_{j,k})\in\textnormal{M}_{n} (\lmodule),  \\
& \tr_{\textnormal{M}(\rmodule)} : \textnormal{M}_{n} (\rmodule) \to \C , \ \tr_{\textnormal{M}(\rmodule)}\big( (\mathbf{b}_{j,k})\big) = \frac{1}{\NumWin} \sum_{j} \tr_{\rmodule} (\mathbf{b}_{j,j}), \ \ (\mathbf{b}_{j,k})\in\textnormal{M}_{n}(\rmodule).\end{align*}
Observe that
\begin{equation} \tr_{\textnormal{M}(\lmodule)} \big( \gmlhs{(f_{j})}{(g_{j})}{\La} \big)  = \tr_{\textnormal{M}(\rmodule)} \big( \gmrhs{(g_{j})}{(f_{j})}{\Lac} \big) = \sum_{j\in\Z_{\NumWin}} \langle f_{j},g_{j}\rangle \ \ \text{for all} \ \ (f_{j}),(g_{j})\in \SO(G)^{n}. \end{equation}
The matrix valued inner products allow us to define module norms on $\SO(G)^{n}$, 
\[ \|(g_{j})\|_\Lambda=\big\Vert \gmlhs{(g_{j})}{(g_{j})}{\La}\big\Vert^{1/2}_{\textnormal{op},L^{2}}, \ \|g\|_{\Lambda^\circ}=\big\Vert\gmrhs{(g_{j})}{(g_{j})}{\Lac}\big\Vert^{1/2}_{\textnormal{op},L^{2}}.\]
The family of functions $\{ \pi (\lambda) g_{j} \}_{\lambda \in \Lambda,j\in\Z_{n}}$ generated by an $n$-tuple $(g_j)$ in $\SO(G)^n$ and by a closed sungroup $\Lambda$ in $G\times\ghat$ is a  \emph{multi-window Gabor system} . 

\begin{definition} \label{def:gabor_frame}
A multi-window Gabor system $\{ \pi (\lambda) g_{j} \}_{\lambda \in \Lambda,j\in\Z_{n}}$ is a \emph{Gabor frame} for $L^2(G)$ if there exists constants $A, B > 0$ such that either of the following equivalent conditions are satisfied.
\begin{enumerate}
\item[(i)] For all $f\in L^{2}(G)$ 
\begin{align} \label{eq:framecondition} \textstyle
A \| f \|_{L^2(G)}^2 \leq \int_{\Lambda} |\langle f, \pi(\lambda) g \rangle |^2 \, d\mu_{\Lambda} (\lambda) \leq B \| f \|_{L^2(G)}^2.
\end{align}  
\item[(ii)]For all $f\in \SO(G)$\[ \textstyle A \, \tr_{\lmodule} \big( \glhs{f}{f}{\La}\big) \le \sum\limits_{j\in\Z_{n}} \tr_{\lmodule} \big(\glhs{f}{g_{j}}{\La}\glhs{g_{j}}{f}{\La}\big) \le B \, \tr_{\lmodule} \big( \glhs{f}{f}{\La}\big).\]
\item[(iii)] For all $(f_{j})\in \SO(G)^{n}$
\[ \textstyle A \, \tr_{\textnormal{M}(\lmodule)} \big( \gmlhs{(f_{j})}{(f_{j})}{\La}\big) \le \tr_{\textnormal{M}(\lmodule)} \big(\gmlhs{(f_{j})}{(g_{j})}{\La}\gmlhs{(g_{j})}{(f_{j})}{\La}\big) \le B \, \tr_{\textnormal{M}(\lmodule)} \big( \gmlhs{(f_{j})}{(f_{j})}{\La}\big).\]
\end{enumerate}
The constants $A$ and $B$ are called \emph{lower} and \emph{upper frame bounds}, respectively. 
\end{definition}
The largest possible value for $A$, $A_{\textnormal{opt}}$, is the optimal lower frame bound and the smallest possible valued for $B$, $B_{\textnormal{opt}}$, is the optimal upper frame bound. As shown in Lemma 3.6 and Remark 3.13 of \cite{jalu18}, $B_{\textnormal{opt}} = \Vert (g_{j}) \Vert_{\La}$.
A Gabor frame for which the frame bounds are equal is called \emph{tight};  it is called \emph{Parseval} if its frame bound equals 1. 

Necessary conditions for $(g_{j})$ and $\La$ to generate a Gabor frame are that the group $(G\times\ghat)/\La$ is compact and that $A \, \textnormal{s}(\La) \le \sum_{j} \Vert g_{j} \Vert_{2}^{2} \le B \, \textnormal{s}(\La)$. In case $\La$ is discrete and equipped with the counting measure, then furthermore, the condition $\textnormal{s}(\La)<n$ is necessary.

The Gabor system $\{\pi(\lambda) g_{j}\}_{\la\in\La,j\in\Z_n}$ is a Gabor frame for $L^{2}(G)$ if and only if the associated \emph{frame operator} $S_{(g_{j}),\La}$ is a continuous invertible bijection on both $L^{2}(G)$ and $\SO(G)$ \cite{grle04}. The frame operator is given by any of the following expressions, for any $f\in\SO(G)$,
\begin{align*} S_{(g_{j}),\La} (f) & := \sum_{j\in\Z_{n}}\int_{\La} \langle f, \pi(\la) g_{j} \rangle \pi(\la) g_{j} \, d\la = \sum_{j\in\Z_{n}}\glhs{f}{g_{j}}{\La} \cdot g_{j} \\
& =  f \cdot \sum_{j\in\Z_{n}} \grhs{g_{j}}{g_{j}}{\Lac} = \frac{1}{\textnormal{s}(\La)} \sum_{j\in\Z_n} \sum_{\lac\in\Lac} \langle g_{j} , \pi(\lac)^{*} g_{j}\rangle \pi(\lac)^{*} f.\end{align*}

If $(g_{j})\in\SO(G)^{n}$ and $\La$ generate a Gabor frame for $L^{2}(G)$, then there exist  functions $(h_{j})\in\SO(G)^{n}$ such that the following equivalent statements hold,
\begin{enumerate}
    \item[(i)] $f = \int_{\La} \langle f, \pi(\lambda) g_{j} \rangle \pi(\la) h_{j}  \, d\la$ for all $f\in L^{2}(G)$,
    \item[(ii)] $f = \sum_{j\in\Z_{n}}\glhs{f}{g_{j}}{\La}\cdot h_{j}$ for all $f\in \SO(G)$,
    \item[(iii)] $(f_{j}) = \gmlhs{(f_{j})}{(g_{j})}{\La}\cdot (h_{j})$ for all $(f_{j})\in \SO(G)^{n}$.
\end{enumerate}
In that case we say that $(g_{j})$ and $(h_{j})$ are a \emph{dual pair} of Gabor frame generators and that $\{\pi(\la) g_{j}\}_{\la\in\La,j\in\Z_{n}}$ and $\{\pi(\la) g_{j}\}_{\la\in\La,j\in\Z_{n}}$ are \emph{dual} Gabor frames for $L^{2}(G)$. The canonical choice of the functions $(h_{j})$ is the \emph{canonical dual frame}: the Gabor frame generated by the functions $(h_{j}) = (S_{(g_{j}),\La}^{-1} g_{j})$.

The following is an adaptation of Corollary 3.15 in \cite{jalu18}.
\begin{lemma} \label{le:wex-raz} Let $\Lambda$ be a closed co-compact subgroup of $G\times\ghat$ and let $(g_{j})$ and $(h_{j})$ be $n$-tuples in $\SO(G)^{n}$. The following statements are equivalent.
\begin{enumerate}[(i)]
\item $f = \sum\limits_{j\in\Z_{\NumWin}} \glhs{f}{\vvfun{g}{j}}{\La}\cdot \vvfun{h}{j}$ for all $f\in\SO(G)$.
\item $\sum\limits_{j\in\Z_{\NumWin}} \grhs{\vvfun{g}{j}}{\vvfun{h}{j}}{\Lac}$ is the identity operator on $L^{2}(G)$.
\item $({g}_{j})$ and $({h}_{j})$ generate dual multi-window Gabor frames with respect to $\La$ for $L^{2}(G)$.
\item the $\lmodule$-valued $\NumWin\times\NumWin$-matrix
\[ \gmlhs{(g_{j})}{(h_{j})}{\La} = \begin{bmatrix} \glhs{\vvfun{g}{1}}{\vvfun{h}{1}}{\La} & \glhs{\vvfun{g}{1}}{\vvfun{h}{2}}{\La} & \cdots & \glhs{\vvfun{g}{1}}{\vvfun{h}{\NumWin}}{\La} \\ 
\glhs{\vvfun{g}{2}}{\vvfun{h}{1}}{\La} & \glhs{\vvfun{g}{2}}{\vvfun{h}{2}}{\La} & \cdots & \glhs{\vvfun{g}{2}}{\vvfun{h}{\NumWin}}{\La} \\
\vdots & \vdots & \ddots & \vdots \\
\glhs{\vvfun{g}{\NumWin}}{\vvfun{h}{1}}{\La} & \glhs{\vvfun{g}{\NumWin}}{\vvfun{h}{2}}{\La} & \cdots & \glhs{\vvfun{g}{\NumWin}}{\vvfun{h}{\NumWin}}{\La}
\end{bmatrix}\]
is an idempotent operator from $L^{2}(G)^{n}$ onto $
V := \overline{\textnormal{span}}\big\{ \oplus_{j\in\Z_{\NumWin}} \pi(\lambda^{\circ})^{*} \vvfun{g}{j} \big\}_{\lambda^{\circ}\in\La^{\circ}}$.
\end{enumerate}
\end{lemma}

Given an $n$-tuple of functions $(g_{j})\in\SO(G)^{n}$ and a closed subgroup $\La$, we define the constant
\[ B((g_{j}),\La) := \frac{1}{\textnormal{s}(\La)} \sum_{j\in\Z_{n}} \Vert \grhs{g_{j}}{g_{j}}{\Lac} \Vert_{\rmodule} = \frac{1}{\textnormal{s}(\Lambda)} \sum_{j\in\Z_{n}} \sum_{\lac\in\Lac} \big\vert \langle g_{j}, \pi(\lac)^{*} g_{j}\rangle\big\vert \]
The properties of $\SO$ imply that this quantity in fact is finite.
\begin{lemma} \label{le:so-implies-bessel} For any $n$-tuple of functions $(g_{j})\in \SO(G)^{n}$ and closed subgroup $\Lambda$ of $G\times\ghat$ the Gabor system $\{\pi(\lambda) g_{j} \}_{\lambda\in\Lambda,j\in\Z_{n}}$ satisfies the upper frame inequality. In fact,  
\[ \Vert (g_{j}) \Vert_{\La} = \Vert (g_{j})\Vert_{\Lac} = B_{\textnormal{opt}} \le B((g_{j}),\La).\] \end{lemma}
\begin{proof}
This follows from the proof of Lemma 4.26 in \cite{jalu18}.
\end{proof}

\begin{lemma} \label{le:frame-bound-from-dual-system} Let $(g_{j})$ and $(h_{j})$ be functions in $\SO(G)^{n}$ that generate dual Gabor frames for $L^{2}(G)$. If the Gabor system generated by $(h_{j})$ has an upper frame bound $B_{h}$, then $B_{h}^{-1}$ is a lower frame bound for the Gabor system generated by $(g_{j})$.
\end{lemma}
\begin{proof}
This is a general result of frame theory and can be found in, e.g, \cite{ch16}. 
\end{proof}

\begin{lemma} \label{le:canonical-dual-form} Let $(g_{j})\in\SO(G)^{n}$ and $\La$ be a closed subgroup of $G\times\ghat$ such that the Gabor system $\{\pi(\lambda) g_j \}_{\lambda\in\Lambda,j\in\Z_{n}}$ is a Gabor frame for $L^{2}(G)$. The canonical dual generators $(h_{j}) = (S_{(g_{j}),\La}^{-1}g_{j})$ are the unique dual generators that lie in the closed subspace of $L^{2}(G)$ given by $\overline{\textnormal{span}} \{\pi(\lambda^{\circ})^{*} g_{j} \}_{\lac\in\Lac,j\in\Z_{n}}$.
\end{lemma}
\begin{proof} See Lemma 4.15 in \cite{jalu18}.
\end{proof}

\section{Sampling of Gabor frames}
\label{sec:sampling}
Let $(g_{j})$ and $(h_{j})$ be $n$-tuples in $\SO(G)^{n}$ and let $\La$ be some closed co-compact subgroup of $G\times\ghat$. 

We show that, if $\{\pi(\lambda) g_{j}\}_{\lambda\in\Lambda,j\in\Z_{n}}$ and $\{\pi(\la) h_{j}\}_{\la\in\La,j\in\Z_{n}}$ are dual multi-window Gabor frames for $L^{2}(G)$, or equivalently $\gmlhs{(g_{j})}{(h_{j})}{\La}$ is an idempotent element of $\textnormal{M}_{n}(\lmodule)$, then, under certain assumptions, the functions obtained by restriction of the generators to a closed subgroup $H$ of $G$ preserves these properties.

In order to formulate the result we need a way to think of elements in the time-frequency plane of $H$, $H\times\hhat$, as elements of the time-frequency plane of $G$, $G\times\ghat$. We do this by constructing an injection from $H\times\hhat$ into $G\times\ghat$:
\begin{remark}\label{rem:1901} Let $H$ be a closed subgroup of $G$. 
Note that $\hhat$ can be identified (as a toplogical group) with the quotient group $\ghat/H^{\perp}$. This quotient group has a set of coset representatives $K_{H^{\perp}}$ in $\ghat$. If we \emph{fix} such a set of coset representatives, then every coset in $\ghat/H^{\perp}$ has a unique representation as $k + H^{\perp}$, where $k\in K_{H^{\perp}}$. This establishes a bijection between $\ghat/H^{\perp}$ and $K_{H^{\perp}}$.  
Due to the isomorphism between $\ghat/H^{\perp}$ and $\hhat$ we can define a bijection
\[ \phi : \hhat \to K_{H^{\perp}}\subseteq \ghat, \ \phi(\omega) = k.\]
For our purposes we will always take $K_{H^{\perp}}$ so that $0\in K_{H^{\perp}}$. This is always possible. Observe that this implies that $\phi(0) = 0$. 
For any given character $\omega\in\hhat$ the element $\phi(\omega)$ is an extension of $\omega$ to a character on $G$. It is clear that this extension crucially depends on the choice of $K_{H^{\perp}}$. 

With the identification $\phi$ between $\hhat$ and $K_{H^{\perp}}$ we construct the injective operator 
\[ \Phi : H \times\hhat \to H\times K_{H^{\perp}}\subseteq G \times\ghat, \ \Phi( x, \omega) = \big( x, \phi(\omega) \big) , \ x\in H, \ \omega\in\hhat.\]
This operator allows us to regard elements of $H\times\hhat$ as elements of $G\times\ghat$. Observe that $\Phi(0) = 0$. Furthermore, for any $\chi = (x,\omega)\in H\times\hhat$ and any $f\in \SO(G)$
\begin{equation} \label{eq:pi-phi-R-intertwining} \pi(\chi) \restrict{H} f = \restrict{H} \pi(\Phi(\chi)) f.\end{equation}
Note that the time-frequency shift on the left acts on functions on $\SO(H)$, where as to the right, the time-frequency shift acts on functions in $\SO(G)$.

A side remark: it is possible to take $K_{H^{\perp}}$ to be a measurable subset of $\ghat$ (we do not require this extra property for our purposes). In that case $\phi$ is a measure preserving map between the measure spaces $\ghat/H^{\perp}$ (with its Haar measure) and $K_{H^{\perp}}$ (with the measure it inherits from $\ghat$). We refer to \cite[Section 3]{boro15} for more details on this.
\end{remark}

In the following we denote the Banach algebra generated by time-frequency shifts of a subgroup $\Lambda$ of $G\times\ghat$ by $\lmodule^{G}$, and the Banach algebra generated by time-frequency shifts of a subgroup $\tilde{\La}$ of $H\times\hhat$ by $\lmodule^{H}$.

The sampling theorem for Gabor frames and generators of Heisenberg modules reads as follows.

\begin{theorem}\label{th:sampl-non-sep} Let $\Lambda$ be a closed co-compact subgroup of $G\times\ghat$, and let $(g_{j})$ and $(h_{j})$ be $n$-tuple in $\SO(G)^{n}$. Assume that $\gmlhs{(g_{j})}{(h_{j})}{\La}$ is an idempotent element of $\textnormal{M}_{n}(\lmodule^{G})$, i.e., $\{\pi(\la) g_{j}\}_{\la\in\La,j\in\Z_{n}}$ and $\{\pi(\la) h_{j}\}_{\la\in\La,j\in\Z_{n}}$ are dual Gabor frames for $L^{2}(G)$.
If the following two assertions are satisfied,
\begin{itemize}
\item[(i)] $H$ is a closed cocompact subgroup of $G$ such that $\Lambda \subseteq H\times\ghat$,
\item[(ii)] $\tilde{\Lambda}$ is a closed cocompact subgroup of $H\times\widehat{H}$ such that $\Phi(\tilde{\Lambda}^{\circ}) \subseteq \Lambda^{\circ}$,
\end{itemize}
then the $n$-tuple
$(\tilde{g}_{j})$ and $(\tilde{h}_{j})$ in $\SO(H)^{n}$ given by $\tilde{g}_{j} = c\, \restrict{H}g_{j}$, $\tilde{h}_{j} = c\, \restrict{H}h_{j}$, and where 
\[ c= \big( \textnormal{s}_{H\times\hhat}(\tilde{\Lambda}) \, \textnormal{s}_{G}(H) / \textnormal{s}_{G\times\ghat}(\Lambda)\big)^{1/2},\] 
are such that $\gmlhs{(\tilde{g}_{j})}{(\tilde{h}_{j})}{\widetilde{\La}}$ is an idempotent element of $\textnormal{M}_{n}(\lmodule^{H})$, i.e., the two Gabor systems $\{\pi(\la) \tilde{g}_{j}\}_{\la\in\widetilde{\La},j\in\Z_{n}}$ and $\{\pi(\la) \tilde{h}_{j}\}_{\la\in\widetilde{\La},j\in\Z_{n}}$ are dual frames for $L^{2}(H)$.
Moreover, the optimal frame bounds $A_{\textnormal{opt}}$ and $B_{\textnormal{opt}}$ of the Gabor frame $\{\pi(\la) \tilde{g}_{j}\}_{\la\in\widetilde{\La},j\in\Z_{n}}$ satisfy the estimate
\begin{equation} \label{eq:1210a} B((h_{j}),\La)^{-1} \le A_{\textnormal{opt}} \le B_{\textnormal{opt}} = \Vert (\tilde{g}_{j}) \Vert_{\widetilde{\La}} \le B((g_{j}),\La).\end{equation}
\end{theorem}
\begin{remark} The merit of the inequalities \eqref{eq:1210a} is that it shows that the condition number (the ratio between the optimal upper and lower frame bound) of \emph{any} new Gabor frame for $L^{2}(H)$ generated by $(\tilde{g}_{j})$ and $\tilde{\La}$ obtained via Theorem \ref{th:sampl-non-sep} is bounded by $B((g_{j}),\La) \cdot B((h_{j}),\La)$.
\end{remark}
\begin{corollary}[Oversampling] Consider the situation as in Theorem \ref{th:sampl-non-sep}. If $\tilde{\Lambda}$ is a subgroup of $G\times\ghat$ such that $\Lambda\subseteq \tilde{\Lambda}$, then the functions
$(\tilde{g}_{j})$ and $(\tilde{h}_{j})$ in $\SO(G)^{n}$ given by $\tilde{g}_{j} = c\, g_{j}$, $\tilde{h}_{j} = c\, h_{j}$, and where 
\[ c= \big( \textnormal{s}_{G\times\ghat}(\tilde{\Lambda}) / \textnormal{s}_{G\times\ghat}(\Lambda)\big)^{1/2},\] 
are such that $\gmlhs{(\tilde{g}_{j})}{(\tilde{h}_{j})}{\widetilde{\La}}$ is an idempotent element of $\textnormal{M}_{n}(\lmodule^{G})$, i.e., the Gabor systems $\{\pi(\la) \tilde{g}_{j}\}_{\la\in\widetilde{\La},j\in\Z_{n}}$ and $\{\pi(\la) \tilde{h}_{j}\}_{\la\in\widetilde{\La},j\in\Z_{n}}$ are dual frames for $L^{2}(G)$.
\end{corollary}
\begin{proof} Apply Theorem \ref{th:sampl-non-sep} with $H=G$ and $\tilde{\Lambda}$ such that $\Lambda\subset\tilde{\Lambda}$. In that case $\Phi$ is the identity operator on $G\times\ghat$. It is easy to verify that conditions (i) and (ii) in Theorem \ref{th:sampl-non-sep} are satisfied. The statement follows.
\end{proof}

In general, the assumptions in Theorem \ref{th:sampl-non-sep} do \emph{not} guarantee that the sampling procedure preserves canonical pairs of dual frames. 
The following lemma provides a sufficient condition for this.
\begin{proposition} \label{pr:canonical-preserved} If condition (ii) in Theorem \ref{th:sampl-non-sep} is strengthened to be
\begin{itemize}
\item[(ii*)] $\tilde{\Lambda}$ is a closed cocompact subgroup of $H\times\widehat{H}$ such that $\Phi(\tilde{\Lambda}^{\circ}) = \Lambda^{\circ} \cap (G\times K_{H^{\perp}})$,
\end{itemize}
then the process described in Theorem \ref{th:sampl-non-sep} preserves pairs of canonical dual frames. That is, if $(g_{j})$ is an $n$-tuple in $\SO(G)^{n}$ that generates a Gabor frame for $L^{2}(G)$ with respect to time-frequency shifts from $\La$ and $(h_{j}) = (S_{(g_{j}),\La}^{-1} g_{j})$, then the dual pair of Gabor frame generators $(\tilde{g}_{j})$ and $(\tilde{h}_{j})$ constructed by Theorem \ref{th:sampl-non-sep} are such that $(\tilde{h}_{j})=(S_{\tilde{g}_{j},\widetilde{\La}}^{-1} \tilde{g}_{j})$.
\end{proposition}

Let us give the proof of Theorem \ref{th:sampl-non-sep} and Proposition \ref{pr:canonical-preserved}. The proof of Theorem \ref{th:sampl-non-sep} builds on the ideas of S\o ndergaard \cite{so07}. The following lemma is essential.

\begin{lemma} \label{lem:sampledframes-2}
For any closed cocompact subgroup $H$ of $G$ and any two functions $f_{1}, f_{2} \in\SO(G)$
\[ 
\big\langle \restrict{H}f_{1}, \restrict{H}f_{2} \big\rangle_{L^2 (H)} = \frac{1}{\textnormal{s}_{G}(H)} \, \sum_{\gamma\in H^{\perp}} \big\langle f_{1}, E_{\gamma} f_{2} \big\rangle_{L^2 (G)}.\]
\end{lemma}
\begin{proof} 
It is clear that $\big\langle \restrict{H}f_{1}, \restrict{H}f_{2} \big\rangle_{L^2(H)} = \int_H \big(f_{1} \cdot \overline{f_{2}}\big)(t) \, d\mu_H (t)$.
Since $f_{1} \cdot \overline{f_{2}}$ is a function in $\SO(G)$ we may apply Poisson's formula as in \eqref{eq:poisson-for-cocompact}. This yields the desired equality
\begin{align*}
\big\langle \restrict{H}f_{1}, \restrict{H}f_{2} \big\rangle_{L^2(H)} &= \frac{1}{\textnormal{s}_{G}(H)} \, \sum_{\gamma\in H^{\perp}} \mathcal{F}( f_{1} \cdot \overline{f_{2}})(\gamma)  \\ 
&= \frac{1}{\textnormal{s}_{G}(H)} \, \sum_{\gamma\in H^{\perp}} \int_G \big(f_{1} \cdot \overline{f_{2}}\big)(t) \, \overline{\gamma(t)} \, d\mu_G (t)  \\
&= \frac{1}{\textnormal{s}_{G}(H)} \, \sum_{\gamma\in H^{\perp}} \int_G f_{1}(t) \cdot \overline{\big( E_{\gamma} f_{2}\big)(t)} \,  d\mu_G (t)  \\
&= \frac{1}{\textnormal{s}_{G}(H)} \, \sum_{\gamma\in H^{\perp}} \big\langle f_{1} , E_{\gamma} \, f_{2}\big\rangle_{L^{2}(G)}.\end{align*}
\end{proof}

\begin{proof}[Proof of Theorem \ref{th:sampl-non-sep}]
In order for the functions $(\tilde{g}_{j})$ and $(\tilde{h}_{j})$ in $\SO(H)^{n}$ to form a pair of dual multi-window Gabor frames for $L^{2}(H)$ with respect to time-frequency shifts of $\tilde{\Lambda}$, i.e., $\gmlhs{(\tilde{g}_{j})}{(\tilde{h}_{j})}{\widetilde{\La}}$ is an idempotent operator in $\textnormal{M}_{n}(\lmodule^{H})$, it is, by Lemma \ref{le:wex-raz}, necessary and sufficient that $\gmrhs{(\tilde{g}_{j})}{(\tilde{h}_{j})}{\widetilde{\La}^{\circ}}$ is the identity operator on $L^{2}(H)^{n}$. That is, 
\begin{equation} \label{eq:1911a} \sum_{j\in\Z_{n}} \big\langle \tilde{h}_{j}, \pi(\lambda^{\circ})^{*} \tilde{g}_{j}\big\rangle_{L^{2}(H)} = s_{H\times\hhat}(\tilde{\Lambda}) \, \delta_{{\lambda}^{\circ},0} \ \ \text{for all} \ \ {\lambda}^{\circ}\in \tilde{\Lambda}^{\circ}.\end{equation}
We use that $\tilde{g}_{j}= c \, \restrict{H}g_{j}$, $\tilde{h}_{j} = c \, \restrict{H} h_{j}$  to rewrite the left hand side of \eqref{eq:1911a},
\begin{align*} & \sum_{j\in\Z_{n}} \big\langle \tilde{h}_{j}, \pi(\lambda^{\circ})^{*} \tilde{g}_{j}\big\rangle_{L^{2}(H)} = c^{2} \sum_{j\in\Z_{n}} \big\langle \restrict{H}h_{j}, \pi(\lambda^{\circ})^{*} \restrict{H}g_{j}\big\rangle_{L^{2}(H)} \\
& \stackrel{\eqref{eq:pi-phi-R-intertwining}}{=} c^{2} \sum_{j\in\Z_{n}} \big\langle \restrict{H}h_{j}, \restrict{H} \pi(\Phi(\lambda^{\circ}))^{*} g_{j}\big\rangle_{L^{2}(H)} \\
& = \frac{c^2}{s_{G}(H)} \, \sum_{\gamma\in H^{\perp}} \sum_{j\in\Z_{n}}  \big\langle h_{j}, \pi(0, \gamma)^{*} \pi(\Phi(\lambda^{\circ}))^{*} g_{j} \big\rangle_{L^2 (G)}.\end{align*}
In the last step  Lemma \ref{lem:sampledframes-2} is used. Hence \eqref{eq:1911a} becomes
\begin{equation} \label{eq:1911b} 
\frac{c^2}{s_{G}(H)} \, \cocy\big(\Phi(\lambda^{\circ}), (0,\gamma)\big) \, \sum_{\gamma\in H^{\perp}} \sum_{j\in\Z_{n}} \big\langle h_{j},  \pi\big((0, \gamma) + \Phi(\lambda^{\circ})\big)^{*} g_{j} \big\rangle_{L^2 (G)} = s_{H\times\hhat}(\tilde{\Lambda}) \, \delta_{\lambda^{\circ},0} \ \ \text{for all} \ \ \lambda^{\circ}\in \tilde{\Lambda}^{\circ}.
\end{equation}
We now observe the following. 

(a). Assumption (i) is equivalent to the fact that $\{0\}\times H^{\perp} \subseteq \Lambda^{\circ}$. Furthermore by (ii) we have $\Phi(\tilde{\Lambda}^{\circ}) \subseteq \Lambda^{\circ}$. The fact that $\Lambda^{\circ}$ is a group implies that the time-frequency shifts that appear in the inner product on the left side in \eqref{eq:1911b}, $\pi\big((0, \gamma) + \Phi(\lambda^{\circ})\big)^{*}$, are of the form $\pi(\lambda^{\circ})$ for some $\lambda^{\circ}\in \Lambda^{\circ}$.

(b). By assumption, the functions $(g_{j})$ and $(h_{j})$ generate dual Gabor frames for $L^{2}(G)$ with respect to time-frequency shifts in $\Lambda$ and hence they satisfy, by Lemma \ref{le:wex-raz}, 
\[ \sum_{j=1}^{N}\big\langle h_{j},\pi(\lambda^{\circ})^{*} g_{j}\big\rangle_{L^{2}(G)} = s_{G\times\ghat}(\Lambda) \, \delta_{\lambda^{\circ},0} \ \ \text{for all} \ \lambda^{\circ}\in\Lambda^{\circ}.\]
A combination of observation (a) and (b) establishes the following relation for the inner products appearing on the left side of \eqref{eq:1911b}:
\[ \sum_{j\in\Z_{n}} \big\langle h_{j},  \pi\big((0, \gamma) + \Phi(\lambda^{\circ})\big)^{*} g_{j} \big\rangle_{L^2 (G)} = \textnormal{s}_{G\times\ghat}(\Lambda) \, \delta_{[(0,\gamma)+\Phi({\lambda}^{\circ})],0} \ \ \text{for all} \ \ \gamma\in H^{\perp} \text{ and } {\lambda}^{\circ}\in \tilde{\Lambda}^{\circ}. \]
Using this in \eqref{eq:1911b} yields the equality
\begin{equation} \label{eq:1911c}
\frac{c^2 \, \textnormal{s}_{G\times\ghat}(\Lambda)}{s_{G}(H)} \, \cocy\big(\Phi(\lambda^{\circ}), (0,\gamma)\big) \, \sum_{\gamma\in H^{\perp}}  \delta_{[(0,\gamma)+\Phi({\lambda}^{\circ})],0}  = \textnormal{s}_{H\times\hhat}(\tilde{\Lambda}) \, \delta_{{\lambda}^{\circ},0} \ \ \text{for all} \ \ {\lambda}^{\circ}\in \tilde{\Lambda}^{\circ}.
\end{equation}
Recall that $\Phi$ maps $H\times\hhat$ into $G\times K_{H^{\perp}}$, where $K_{H^{\perp}}$ is a set of coset representatives of $H^{\perp}$ and also $\Phi$ is constructed such that $\Phi(0)=0$. Because of this, $(0,\gamma) + \Phi({\lambda}^{\circ})$, $\gamma\in H^{\perp}$, ${\lambda}^{\circ}\in\tilde{\Lambda}^{\circ}$ is equal to zero if and only if both $\gamma = 0$ and ${\lambda}^{\circ}=0$. That is, equation \eqref{eq:1911c} is satisfied whenever $(\gamma,\lambda^{\circ}) \ne (0,0)$. The only case left to verify is thus $(\gamma,\lambda^{\circ}) = (0,0)$, in which case the required equality becomes
\[ \frac{c^{2} \, s_{G\times\ghat}(\Lambda)}{s_{G}(H)} = s_{H\times\hhat}(\tilde{\Lambda}) \ \ \Leftrightarrow \ \ c^2 = \frac{s_{H\times\hhat}(\tilde{\Lambda}) \, s_{G}(H)}{s_{G\times\ghat}(\Lambda)}. \]
This is exactly the way we chose the constant $c$. This shows that the desired equality \eqref{eq:1911a} is satisfied and the first statement of the theorem follows. Concerning the moreover part, we observe that by Lemma \ref{le:so-implies-bessel} the optimal Bessel bound $B_{\textnormal{opt}}$ for the Gabor system $\{\pi(\lambda) \tilde{g}_{j}\}_{\lambda\in\tilde{\Lambda},j \in\Z_{n}}$ in $L^{2}(H)$ can be estimated by
\[ B_{\textnormal{opt}} \le B((\tilde{g}_{j}),\tilde{\La}) = \frac{1}{\textnormal{s}_{H\times\hhat}(\tilde{\Lambda})} \sum_{j\in \Z_{n}}\sum_{{\lambda}^{\circ}\in\tilde{\Lambda}^{\circ}} \big\vert \big\langle \tilde{g}_{j} , \pi({\lambda}^{\circ})^{*} \tilde{g}_{j} \big\rangle_{L^{2}(H)} \big\vert.\]
As previously in the proof, an application of Lemma \ref{lem:sampledframes-2} lets us turn the inner product from one on $L^{2}(H)$ into one on $L^{2}(G)$,
\begin{align*}
B_{\textnormal{opt}} & = \frac{1}{\textnormal{s}_{H\times\hhat}(\tilde{\Lambda})} \sum_{j\in\Z_{n}}\sum_{{\lambda}^{\circ}\in\tilde{\Lambda}^{\circ}} \Big\vert \frac{c^{2}}{\textnormal{s}_{G}(H)} \sum_{\gamma\in H^{\perp}} \big\langle g_{j} , \pi\big((0,\gamma)+\Phi({\lambda}^{\circ})\big)^{*} g_{j}\big\rangle_{L^{2}(G)} \Big\vert \\
& \le \frac{c^{2}}{\textnormal{s}_{H\times\hhat}(\tilde{\Lambda}) \, \textnormal{s}_{G}(H)} \sum_{j\in\Z_{n}} \sum_{{\lambda}^{\circ}\in\tilde{\Lambda}^{\circ}} \sum_{\gamma\in H^{\perp}} \big\vert \big\langle g_{j} , \pi\big((0,\gamma)+\Phi({\lambda}^{\circ})\big)^{*} g_{j} \big\rangle_{L^{2}(G)} \big\vert .
\end{align*}
By construction of $c$, we have $c^{2} ({\textnormal{s}_{H\times\hhat}(\tilde{\Lambda}) \, \textnormal{s}_{G}(H)})^{-1} = \textnormal{s}_{G\times\ghat}(\Lambda)^{-1}$. Moreover, by construction of $\Phi$ and the assumption (i) and (ii) the collection of points $\{ (0,\gamma) + \Phi({\lambda}^{\circ})\}_{\gamma\in H^{\perp}, {\lambda}^{\circ}\in\tilde{\Lambda}^{\circ}}$ is a subset of $\Lambda^{\circ}$ (as also argued earlier in the proof). Thus
\[ B_{\textnormal{opt}} \le \frac{1}{\textnormal{s}_{G\times\ghat}(\Lambda)} \sum_{j\in\Z_{n}} \sum_{\lambda^{\circ}\in\Lambda^{\circ}} \big\vert  \big\langle g_{j}, \pi(\lambda^{\circ})^{*} g_{j}\big\rangle_{L^{2}(G)}\big\vert = B((g_{j}),\La). \]
The lower bound follows by the upper frame bound for the dual system, see Lemma \ref{le:frame-bound-from-dual-system}.
\end{proof}

\begin{proof}[Proof of Proposition \ref{pr:canonical-preserved}] Let $(g_{j})$ be functions in $\SO(G)^{n}$ that generate a Gabor frame $\{\pi(\lambda) g_{j} \}_{\lambda\in\Lambda,j\in\Z_{n}}$ for $L^{2}(G)$, which we turn into a Gabor frame $\{\pi({\lambda}) \tilde{g}_{j}\}_{{\lambda}\in\tilde{\Lambda},j\in\Z_{n}}$ for $L^{2}(H)$ as in Theorem \ref{th:sampl-non-sep}. Let $(h_{j})$ be the canonical dual frame generators. By Lemma \ref{le:canonical-dual-form} we know that they are of the form
\[ h_{j} = \sum_{k\in\Z_{n}}\sum_{\lambda^{\circ}\in \Lambda^{\circ}} d_{j}(\lambda^{\circ},k) \, \pi(\lambda^{\circ}) g_{k}, \ \ j\in\Z_{n}\]
for a certain $d_{j}\in \ell^{1}(\Lambda^{\circ}\times \Z_{n})$. 
Using Theorem \ref{th:sampl-non-sep} these dual generators are turned into dual generators of the Gabor system $\{\pi({\lambda}) \tilde{g}_{j}\}_{{\lambda}\in\tilde{\Lambda},j\in\Z_{n}}$ in $\SO(H)$ by restricting them to the subgroup $H$ and multiplying them with the constant $c$,
\begin{align*} \tilde{h}_{j}=c\,\restrict{H}h_{j} & = c \sum_{k\in\Z_{n}} \sum_{\lambda^{\circ}\in \Lambda^{\circ}} d_{j}(\lambda^{\circ},k) \, \restrict{H} \, \pi(\lambda^{\circ}) g_{k}, \ \ j\in\Z_{n}.\end{align*}
By assumption we have $\Phi(\tilde\Lambda^{\circ}) = \Lambda^{\circ} \cap (G\times K_{H^{\perp}})$. Furthermore, by assumption (i) in Theorem \ref{th:sampl-non-sep} the inclusion $\{0\}\times H^{\perp} \subseteq \Lambda^{\circ}$ holds. This implies that every $\lambda^{\circ}\in \Lambda^{\circ}$ can be written in a unique form $\Phi({\lambda}^{\circ})+(0,\gamma)$ for some ${\lambda}^{\circ}\in \tilde{\Lambda}^{\circ}$ and $\gamma\in H^{\perp}$.  Thus
\begin{equation} \label{eq:2103a} \tilde{h}_{j}=c\,\restrict{H}h_{j} = c\sum_{k\in\Z_{n}} \sum_{{\lambda}^{\circ}\in \tilde{\Lambda}^{\circ}} \sum_{\gamma\in H^{\perp}} d_{j}(\Phi({\lambda}^{\circ})+(0,\gamma),k) \, \restrict{H} \, \pi(\Phi({\lambda}^{\circ})+(0,\gamma)) g_{k}.\end{equation}
Observe that 
\[ \restrict{H} \, \pi(\Phi({\lambda}^{\circ})+(0,\gamma)) g = \restrict{H} \, E_{\gamma} \pi(\Phi({\lambda}^{\circ})) g  = \restrict{H} \pi(\Phi({\lambda}^{\circ})) g.\]
In the last step we used that $\restrict{H} E_{\gamma} f= \restrict{H}f$ for all $f\in \SO(G)$ (since, $E_{\gamma}f(t) = \gamma(t) f(t)= f(t)$ for all $t\in H$ and $\gamma\in H^{\perp}$). Furthermore, 
\[ \restrict{H} \, \pi(\Phi({\lambda}^{\circ})+(0,\gamma)) g = \restrict{H} \pi(\Phi({\lambda}^{\circ})) g \stackrel{\eqref{eq:pi-phi-R-intertwining}}{=} \pi({\lambda}^{\circ}) \restrict{H} g.\]
We now continue in \eqref{eq:2103a} and find that the sampled dual generators are of the form
\[ \tilde{h}_{j}=c\,\restrict{H} h_{j} = \sum_{k\in\Z_{n}}\sum_{{\lambda}^{\circ}\in \tilde{\Lambda}^{\circ}} \Big( \sum_{\gamma\in H^{\perp}} d_{j}(\Phi({\lambda}^{\circ})+(0,\gamma),k) \Big) \, \pi({\lambda}^{\circ}) \, \underbrace{c\,  \restrict{H} g_{k}}_{= \tilde{g}_{j}}.\]
This shows that each $\tilde{h}_{j}$, $j\in\Z_{n}$ lies in $\overline{\textnormal{span}}\{\pi({\lambda}^{\circ}) \tilde{g}_{j} \, : \, \lambda^{\circ}\in \tilde{\Lambda}^{\circ} , j\in\Z_{n}\}$. By Lemma \ref{le:canonical-dual-form} it is the canonical dual generator of $(\tilde{g}_{j})$.
\end{proof}

\section{Periodization of Gabor frames}
\label{sec:periodization}

We will now consider the situation where we periodize the dual generators of a multi-window Gabor frame for $L^{2}(G)$ with respect to a discrete subgroup $H$. Under the right assumptions this will give us a multi-window Gabor frame for $L^{2}(G/H)$. We will deduce the result from Theorem \ref{th:sampl-non-sep} together with the following two facts:
\begin{enumerate} 
\item[(i)] For all $f\in\SO(G)$
\[  \period{H} f = \mathcal{F}_{G/H}^{-1} \restrict{H^{\perp}} \mathcal{F}_{G} f.\]
Here $\mathcal{F}_{G/H}$ is the Fourier transform from $\SO(G/H)$ onto $\SO(H^{\perp})$ and $\mathcal{F}_{G}$ is the Fourier transform from $\SO(G)$ onto $\SO(\ghat)$. 
\item[(ii)] The Fourier transform (as any unitary operator) preserves the frame properties of any collection of functions.
\end{enumerate}
Because of these two facts, the $n$-tuples $(\period{H}g_{j})$ and $(\period{H} h_{j})$ in $\SO(G/H)^{n}$ generate dual Gabor frames for $L^{2}(G/H)$ if and only if the $n$-tuples $( \restrict{H^{\perp}} \mathcal{F}_{G} g_{j})$ and $(\restrict{H^{\perp}} \mathcal{F}_{G} h_{j})$ in $\SO(H^{\perp})^{n}$ generate dual Gabor frames for $L^{2}(H^{\perp})$.  This allows us to transfer the statements about sampling of Gabor frames from the previous section into the setting of periodized Gabor frames. 

Mimicking the situation from before, we require an injective function from the time-frequency plane of $G/H$, $(G/H)\times H^{\perp}$ into $G\times\ghat$. 

\begin{remark}\label{rem:1902} Let $H$ be a discrete subgroup of $G$. 
Let $K_{H}$ be a set of coset representatives of the quotient $G/H$ in $G$ such that $0\in K_{H}$. Every coset in $G/H$ has a unique representation as $k + H$, where $k\in K_{H}$. this defines a bijection $\psi$ between $G/H$ and $K_{H}$
\[ \psi : G/H \to K_{H}\subseteq G, \ \psi(k+H) = k.\] 
With this identification between $G/H$ and $K_{H}$ we construct the injective operator 
\[ \Psi : G/H \times H^{\perp} \to G \times\ghat, \ \Psi( k+H, \gamma) = \big( \psi(k+H), \gamma \big) = (k,\gamma), \ k+H\in G/H, \ \gamma\in H^{\perp}.\]
Observe that $\Psi(0) = 0$. Furthermore, for any $\chi = (x,\omega)\in G/H\times H^{\perp}$ and any $f\in \SO(G)$
\begin{equation} \label{eq:pi-psi-P-intertwining} \pi(\chi) \period{H} f = \period{H} \pi(\Psi(\chi)) f.\end{equation}
\end{remark}

\begin{theorem}\label{th:period-non-sep} Let $\Lambda$ be a closed co-compact subgroup of $G\times\ghat$, and let $(g_{j})$ and $(h_{j})$ be $n$-tuple in $\SO(G)^{n}$. Assume that $\gmlhs{(g_{j})}{(h_{j})}{\La}$ is an idempotent element of $\textnormal{M}_{n}(\lmodule^{G})$, i.e., $\{\pi(\la) g_{j}\}_{\la\in\La,j\in\Z_{n}}$ and $\{\pi(\la) h_{j}\}_{\la\in\La,j\in\Z_{n}}$ are dual Gabor frames for $L^{2}(G)$.
If the following two assertions are satisfied,
\begin{itemize}
\item[(i)] $H$ is a discrete subgroup of $G$ such that $\Lambda \subseteq G\times H^{\perp}$ (or, equiv.\ $H\times\{0\} \subseteq \Lambda^{\circ}$)
\item[(ii)] $\tilde{\Lambda}$ is a closed cocompact subgroup of $G/H\times H^{\perp}$ such that $\Psi(\tilde{\Lambda}^{\circ})\subseteq \Lambda^{\circ}$,
\end{itemize}
then the $n$-tuple
$(\tilde{g}_{j})$ and $(\tilde{h}_{j})$ in $\SO(G/H)^{n}$ given by $\tilde{g}_{j} = c\, \period{H}g_{j}$, $\tilde{h}_{j} = c\, \period{H}h_{j}$, and where 
\[ c= \big( \textnormal{s}_{G/H\times H^{\perp}}(\tilde{\Lambda}) \, \textnormal{s}_{\ghat}(H^{\perp}) / \textnormal{s}_{G\times\ghat}(\Lambda) \, \big)^{1/2},\]
are such that $\gmlhs{(\tilde{g}_{j})}{(\tilde{h}_{j})}{\widetilde{\La}}$ is an idempotent element of $\textnormal{M}_{n}(\lmodule^{G/H})$, i.e., the two Gabor systems $\{\pi(\la) \tilde{g}_{j}\}_{\la\in\widetilde{\La},j\in\Z_{n}}$ and $\{\pi(\la) \tilde{h}_{j}\}_{\la\in\widetilde{\La},j\in\Z_{n}}$ are dual frames for $L^{2}(G/H)$.
Moreover, the optimal frame bounds $A_{\textnormal{opt}}$ and $B_{\textnormal{opt}}$ of the Gabor frame $\{\pi(\la) \tilde{g}_{j}\}_{\la\in\widetilde{\La},j\in\Z_{n}}$ satisfy the estimate
\begin{equation} \label{eq:1210aa} B((h_{j}),\La)^{-1} \le A_{\textnormal{opt}} \le B_{\textnormal{opt}} = \Vert (\tilde{g}_{j}) \Vert_{\widetilde{\La}} \le B((g_{j}),\La).\end{equation}
\end{theorem}

As with the sampling theorem, also the assumptions in Theorem \ref{th:period-non-sep} do \emph{not} guarantee that the periodizing procedure preserves canonical pairs of dual frames. The following lemma yields a sufficient condition for this.

\begin{lemma} \label{le:canonical-preserved} If condition (ii) in Theorem \ref{th:period-non-sep} is strengthened to be
\begin{itemize}
\item[(ii*)] $\tilde{\Lambda}$ is a closed cocompact subgroup of $H\times\widehat{H}$ such that $\Psi(\tilde{\Lambda}^{\circ}) = \Lambda^{\circ} \cap (K_{H} \times \ghat)$,
\end{itemize}
then the process described in Theorem \ref{th:period-non-sep} preserves pairs of canonical dual frames. 
\end{lemma}

An alternative proof of these results is to mimic the proof of Theorem \ref{th:sampl-non-sep}, where the crucial Lemma \ref{lem:sampledframes-2} is replaced by the following.

\begin{lemma} For any discrete subgroup $H$ of $G$ and any two functions $f_{1},f_{2} \in\SO(G)$
\begin{equation} 
\big\langle \period{H}f_{1}, \period{H}f_{2} \big\rangle_{L^2 (G/H)} = \frac{1}{\textnormal{s}_{\ghat}(H^{\perp})} \, \sum_{\gamma\in H} \big\langle f_{1}, T_{\gamma} f_{2} \big\rangle_{L^2 (G)}.
\end{equation}
\end{lemma}

\section{Examples}
\label{sec:examples}

In this section we consider several examples. In each of them we detail how we pick the groups $G$, $H$, $\Lambda$ and $\tilde{\Lambda}$, how their respective measures are defined, and we specify how we construct the Banach algebras $\lmodule$ and $\rmodule$.

\subsection{Sampling a separable Gabor frame from $L^{2}(\R^{d})$ to a frame for $\ell^{2}(\Z^{d})$}

Corollary \ref{cor:ex1} below follows by an application of Theorem \ref{th:sampl-non-sep} with the following setup: $G=\ghat=\R^{d}$ (equipped with the Lebesgue measure), $\omega\in \R^{d}$ acts as a character on $x\in \R^{d}$ by $x\mapsto e^{2\pi i x\cdot\omega}$, $H=\alpha a^{-1}\Z^{d}$ (equipped with the counting measure), $\widehat{H}=[0,a \, \alpha^{-1})^{d}$ (equipped with its probability measure and addition modulo $a \, \alpha^{-1}\Z^{d}$), any $\omega\in[0,a \, \alpha^{-1})^{d}$ acts as a character on $x\in \alpha a^{-1}\Z^{d}$ by $x\mapsto e^{2\pi i x\cdot\omega}$. We take the groups $\La,\tilde{\La},\Lac$ and $\tilde{\La}^{\circ}$ to be the lattices (equipped with the counting measure) given by
\begin{align*} 
& \La = \alpha\Z^{d}\times\beta\Z^{d}, \ \ \tilde{\Lambda} = \begin{bmatrix} \alpha & 0 \\ 0 & \beta \end{bmatrix} \cdot \begin{bmatrix} \Z^{d} \\ \{0,1,\ldots,M-1\}^{d}\end{bmatrix},  \\
& \Lac = \begin{bmatrix} 1/\beta & 0 \\ 0 & 1/\alpha \end{bmatrix}\cdot \begin{bmatrix} \Z^{d} \\ \Z^{d}\end{bmatrix} \ \ \text{and} \ \ \tilde{\La}^{\circ} = \begin{bmatrix} 1/\beta & 0 \\ 0 & 1/\alpha \end{bmatrix}\cdot \begin{bmatrix} \Z^{d} \\ \{0,1,\ldots,a-1\}^{d}\end{bmatrix}. \end{align*}
Observe, $\textnormal{s}_{G\times\ghat}(\La) = \textnormal{s}_{H\times\widehat{H}}(\tilde{\La}) = (\alpha\beta)^{d} = (a/M)^{d}$ and $\textnormal{s}_{G}(H) = (\alpha \, a^{-1})^{d}$.
With the notation as in Remark \ref{rem:1901}, we take $K_{H^{\perp}} = [0,a \, \alpha^{-1})^{d}$. Hence the map $\Phi: H\times\widehat{H}=\alpha a^{-1}\Z^{d}\times [0,a \, \alpha^{-1})^{d} \to G\times\ghat=\R^{2d}$ is the one given by $\Phi(x,\omega) = (x,\omega)$, $(x,\omega)\in \alpha a^{-1}\Z^{d}\times [0,a \, \alpha^{-1})^{d}$. One can now verify that the conditions of Theorem \ref{th:sampl-non-sep} and Proposition \ref{pr:canonical-preserved}
are satisfied.
In this case the Banach algebras $\lmodule^{\R^{d}}$ and $\rmodule^{\R^{d}}$ are given by 
\begin{align*} 
\lmodule^{\R^{d}} & = \big\{ {\bf a} \in \mathsf{B}(L^{2}(\R^{d})) \, : \, {\bf{a}} = \sum_{\la\in\La} \mathsf{a}(\lambda) \, \pi(\lambda) , \ \mathsf{a}\in \ell^{1}(\Lambda)\big\}, \\
\rmodule^{\R^{d}} & = \big\{ {\bf b} \in \mathsf{B}(L^{2}(\R^{d})) \, : \, {\bf{b}} = \frac{1}{\alpha\beta} \sum_{\lac\in\Lambda^{\circ}} \mathsf{b}(\lambda^{\circ}) \, \pi(\lambda^{\circ})^{*}, \ \mathsf{b}\in \ell^{1}(\Lambda^{\circ})\big\}, \end{align*}
and $\Vert {\bf a} \Vert_{\lmodule^{\R^{d}}} = \Vert \mathsf{a} \Vert_{1}$ and $\Vert {\bf b} \Vert_{\lmodule^{\R^{d}}} = \Vert \mathsf{b} \Vert_{1}$. Observe that $\lmodule^{\R^{d}}$ is generated by the two unitary operators $U$ and $V$ on $L^{2}(\R^{d})$ given by  
\[ Uf(t) = T_{\alpha} f(t) = f(t-\alpha) \ \ \text{and} \ \ Vf(t) = E_{\beta} f(t) = e^{2\pi i \beta t} f(t), \ \ t\in\R^{d},\]
where as $\rmodule^{\R^{d}}$ is generated by the two unitaries ${U}^{\circ}$ and ${V}^{\circ}$ given by
\[ U^{\circ} f(t) = E_{1/\alpha} f(t) = e^{2\pi i t/\alpha} f(t) \ \ \text{and} \ \ V^{\circ} f(t) = T_{1/\beta} f(t) = f(t-1/\beta), \ \ t\in\R^{d}.\]
The Banach algebras $\lmodule^{\alpha a^{-1}\Z^{d}}$ and $\rmodule^{\alpha a^{-1}\Z^{d}}$ are given by
\begin{align*} 
\lmodule^{\alpha a^{-1}\Z^{d}} & = \big\{ {\bf a} \in \mathsf{B}(\ell^{2}(\alpha a^{-1}\Z^{d})) \, : \, {\bf{a}} = \sum_{\la\in\tilde\La} \mathsf{a}(\lambda) \, \pi(\lambda) , \ \mathsf{a}\in \ell^{1}(\tilde\Lambda)\big\}, \\
\rmodule^{\alpha a^{-1}\Z^{d}} & = \big\{ {\bf b} \in \mathsf{B}(\ell^{2}(\alpha a^{-1}\Z^{d})) \, : \, {\bf{b}} = \frac{1}{\alpha\beta} \sum_{\lac\in\tilde\Lambda^{\circ}} \mathsf{b}(\lambda^{\circ}) \, \pi(\lambda^{\circ})^{*}, \ \mathsf{b}\in \ell^{1}(\tilde\Lambda^{\circ})\big\}. \end{align*}
The algebra $\lmodule^{\alpha a^{-1}\Z^{d}}$ is generated by the unitaries $\tilde{U}$ and $\tilde{V}$ on $\ell^{2}(\alpha a^{-1}\Z^{d})$ given by
\[ \tilde{U} f(t) = T_{\alpha}f(t) = f(t-\alpha) \ \ \text{and} \ \ \tilde{V}f(t) = E_{\beta} f(t) = e^{2\pi i \beta t} f(t-q\alpha), \ \ t\in \alpha a^{-1}\Z^{d},\]
and the algebra $\rmodule^{\alpha a^{-1}\Z^{d}}$ is generated by the unitaries $\tilde{U}^{\circ}$ and $\tilde{V}^{\circ}$ on $\ell^{2}(\alpha a^{-1}\Z^{d})$ given by
\[ \tilde{U}^{\circ} f(t) = E_{1/\alpha}  f(t) = e^{2\pi i t/\alpha} f(t) \ \ \text{and} \ \ \tilde{V}^{\circ} f(t) = T_{1/\beta} f(t) = f(t-1/\beta), \ \ t\in \alpha a^{-1}\Z^{d}. \]
Observe that 
\[ VU = e^{2\pi i d\alpha\beta}\, UV\ , \ \ \tilde{V} \tilde{U} = e^{2\pi i d\alpha\beta} \, \tilde{U} \tilde{V}.\]
I.e., $\lmodule^{\R^{d}}$ and $\lmodule^{\alpha a^{-1} \Z^{d}}$ give rise to the same non-commutative torus.
We can formulate the following result.

\begin{corollary} \label{cor:ex1} Let $\La$ be the lattice in $\R^{2d}$ of the form $\La = \alpha \Z^{d}\times \beta \Z^{d}$, $\alpha,\beta>0$,
and let $(g_{j})$ and $(h_{j})$ be $n$-tuple in $\SO(\R^{d})^{n}$.
Assume that $\gmlhs{(g_{j})}{(h_{j})}{\La}$ is an idempotent element of $\textnormal{M}_{n}(\lmodule^{\R^{d}})$, i.e., $\{\pi(\la)g_{j}\}_{\la\in\La,j\in\Z_{n}}$ and $\{\pi(\la)h_{j}\}_{\la\in\La,j\in\Z_{n}}$ are dual Gabor frames for $L^{2}(\R^{d})$. If $\alpha$ and $\beta$ are such that $\alpha\beta = a/M$ for some $a,M\in\N$ and if we define
\[ \tilde{\Lambda} = \begin{bmatrix} \alpha & 0 \\ 0 & \beta \end{bmatrix} \cdot \begin{bmatrix} \Z^{d} \\ \{0,1,\ldots,M-1\}^{d}\end{bmatrix} \subset \alpha \, a^{-1} \Z^{d} \times [0,M\beta)^{d}, \]
then the $n$-tuple 
$(\tilde{g}_{j})$ and $(\tilde{h}_{j})$ in $(\ell^{1}(\alpha a^{-1}\Z^{d}))^{n}$ given by $\tilde{g}_{j} = c\, \restrict{\alpha a^{-1}\Z^{d}}g_{j}$, $\tilde{h}_{j} = c\, \restrict{\alpha a^{-1}\Z^{d}}h_{j}$, and where 
\[ c= \big(  \alpha a^{-1} \big)^{d/2},\] 
are such that $\gmlhs{(\tilde{g}_{j})}{(\tilde{h}_{j})}{\widetilde{\La}}$ is an idempotent element of $\textnormal{M}_{n}(\lmodule^{\alpha a^{-1}\Z^{d}})$, i.e., the two Gabor systems $\{\pi(\la) \tilde{g}_{j}\}_{\la\in\widetilde{\La},j\in\Z_{n}}$ and $\{\pi(\la) \tilde{h}_{j}\}_{\la\in\widetilde{\La},j\in\Z_{n}}$ are dual frames for $\ell^{2}(\alpha a^{-1}\Z^{d})$. Moreover, canonical pairs of dual Gabor frames are preserved by this process. 
\end{corollary}

\subsection{Sampling a non-separable 
Gabor frame for $L^{2}(\R)$ to a frame for $\ell^{2}(\Z)$}
Corollary \ref{cor:ex2} below follows by an application of Theorem \ref{th:sampl-non-sep} with the following setup:
we take $G=\ghat=\R$ (equipped with the Lebesgue measure), $\omega\in \R$ acts as a character on $x\in \R$ by $x\mapsto e^{2\pi i x\omega}$, $H=\alpha (sa)^{-1}\Z$ (equipped with the counting measure), $\widehat{H}=[0,sa \, \alpha^{-1})$ (equipped with its probability measure and addition modulo $sa \, \alpha^{-1}$), $\omega\in[0,sa \, \alpha^{-1})$ acts as a chatacter on $x\in \alpha (sa)^{-1}\Z$ by $x\mapsto e^{2\pi i x\omega}$, the groups $\La,\tilde{\La},\Lac$ and $\tilde{\La}^{\circ}$ are equipped with the counting measure and given by 
\begin{align*} & \Lambda = \begin{bmatrix} \alpha & q\alpha \\ 0 & \beta \end{bmatrix}\cdot\begin{bmatrix} \Z \\ \Z \end{bmatrix}, \ \ \alpha,\beta>0, q\in(0,1), q=r/s, \ r,s\in\N, \\
& \tilde{\Lambda} = \begin{bmatrix} \alpha & q\alpha \\ 0 & \beta \end{bmatrix}\cdot\begin{bmatrix} \Z \\ \{0,1,\ldots,sM-1\} \end{bmatrix} \subset \frac{\alpha}{sa}\Z \times [0,\frac{sa}{\alpha}), \\
&  \Lac = \begin{bmatrix} q/\beta & 1/\beta \\ 1/\alpha & 0 \end{bmatrix}\cdot \begin{bmatrix} \Z \\ \Z\end{bmatrix} \ \ \text{and} \ \ \tilde{\La}^{\circ} = \begin{bmatrix} q/\beta & 1/\beta \\ 1/\alpha & 0 \end{bmatrix}\cdot \begin{bmatrix} \Z \\ \{0,1,\ldots,sa-1\}\end{bmatrix}. \end{align*}
Observe, $\textnormal{s}_{G\times\ghat}(\La) = \textnormal{s}_{H\times\widehat{H}}(\tilde{\La}) = \alpha\beta = a/M$ and $\textnormal{s}_{G}(H) = \alpha (sa)^{-1}$.
With the notation as in Remark \ref{rem:1901}, we take $K_{H^{\perp}} = [0,sa \, \alpha^{-1})$. Hence the map $\Phi: H\times\widehat{H}=\alpha (sa)^{-1}\Z\times [0,sa \, \alpha^{-1}) \to G\times\ghat=\R^{2}$ is the one given by $\Phi(x,\omega) = (x,\omega)$, $(x,\omega)\in \alpha (sa)^{-1}\Z\times [0,sa \, \alpha^{-1})$. One can now verify that the conditions of Theorem \ref{th:sampl-non-sep} and Proposition \ref{pr:canonical-preserved}
are satisfied.
In this case the Banach algebras $\lmodule^{\R}$ and $\rmodule^{\R}$ are given by 
\begin{align*} 
\lmodule^{\R} & = \big\{ {\bf a} \in \mathsf{B}(L^{2}(\R)) \, : \, {\bf{a}} = \sum_{\la\in\La} \mathsf{a}(\lambda) \, \pi(\lambda) , \ \mathsf{a}\in \ell^{1}(\Lambda)\big\}, \\
\rmodule^{\R} & = \big\{ {\bf b} \in \mathsf{B}(L^{2}(\R)) \, : \, {\bf{b}} = \frac{1}{\alpha\beta} \sum_{\lac\in\Lambda^{\circ}} \mathsf{b}(\lambda^{\circ}) \, \pi(\lambda^{\circ})^{*}, \ \mathsf{b}\in \ell^{1}(\Lambda^{\circ})\big\}, \end{align*}
and $\Vert {\bf a} \Vert_{\lmodule^{\R}} = \Vert \mathsf{a} \Vert_{1}$ and $\Vert {\bf b} \Vert_{\lmodule^{\R}} = \Vert \mathsf{b} \Vert_{1}$. Observe that $\lmodule^{\R}$ is generated by the two unitary operators $U$ and $V$ on $L^{2}(\R)$ given by  
\[ Uf(t) = T_{\alpha} f(t) = f(t-\alpha) \ \ \text{and} \ \ Vf(t) = E_{\beta}T_{q\alpha} f(t) = e^{2\pi i \beta t} f(t-q\alpha), \ \ t\in\R,\]
where as $\rmodule^{\R}$ is generated by the two unitaries ${U}^{\circ}$ and ${V}^{\circ}$ given by
\[ U^{\circ} f(t) = E_{1/\alpha} T_{q/\beta} f(t) = e^{2\pi i t/\alpha} f(t-q/\beta) \ \ \text{and} \ \ V^{\circ} f(t) = T_{1/\beta} f(t) = f(t-1/\beta), \ \ t\in\R.\]
The Banach algebras $\lmodule^{\alpha (sa)^{-1}\Z}$ and $\rmodule^{\alpha (sa)^{-1}\Z}$ are given by
\begin{align*} 
\lmodule^{\alpha (sa)^{-1}\Z} & = \big\{ {\bf a} \in \mathsf{B}(\ell^{2}(\alpha (sa)^{-1}\Z)) \, : \, {\bf{a}} = \sum_{\la\in\tilde\La} \mathsf{a}(\lambda) \, \pi(\lambda) , \ \mathsf{a}\in \ell^{1}(\tilde\Lambda)\big\}, \\
\rmodule^{\alpha (sa)^{-1}\Z} & = \big\{ {\bf b} \in \mathsf{B}(\ell^{2}(\alpha (sa)^{-1}\Z)) \, : \, {\bf{b}} = \frac{1}{\alpha\beta} \sum_{\lac\in\tilde\Lambda^{\circ}} \mathsf{b}(\lambda^{\circ}) \, \pi(\lambda^{\circ})^{*}, \ \mathsf{b}\in \ell^{1}(\tilde\Lambda^{\circ})\big\}. \end{align*}
The algebra $\lmodule^{\alpha (sa)^{-1}\Z}$ is generated by the unitaries $\tilde{U}$ and $\tilde{V}$ on $\ell^{2}(\alpha (sa)^{-1}\Z)$ given by
\[ \tilde{U} f(t) = T_{\alpha}f(t) = f(t-\alpha) \ \ \text{and} \ \ \tilde{V}f(t) = E_{\beta}T_{q\alpha} f(t) = e^{2\pi i \beta t} f(t-q\alpha), \ \ t\in \alpha (sa)^{-1}\Z,\]
and the algebra $\rmodule^{\alpha (sa)^{-1}\Z}$ is generated by the unitaries $\tilde{U}^{\circ}$ and $\tilde{V}^{\circ}$ on $\ell^{2}(\alpha (sa)^{-1}\Z)$ given by
\[ \tilde{U}^{\circ} f(t) = E_{1/\alpha} T_{q/\beta} f(t) = e^{2\pi i t/\alpha} f(t-q/\beta) \ \ \text{and} \ \ \tilde{V}^{\circ} f(t) = T_{1/\beta} f(t) = f(t-1/\beta), \ \ t\in \alpha (sa)^{-1}\Z. \]
Observe that 
\[ VU = e^{2\pi i \alpha\beta}\, UV\ , \ \ \tilde{V} \tilde{U} = e^{2\pi i \alpha\beta} \, \tilde{U} \tilde{V}.\]

We can formulate the following result.
\begin{corollary} \label{cor:ex2} Let $\La$ be the lattice in $\R^{2}$ given by
\[ \Lambda = \begin{bmatrix} \alpha & q\alpha \\ 0 & \beta \end{bmatrix}\cdot\begin{bmatrix} \Z \\ \Z \end{bmatrix}, \ \ \alpha,\beta>0, q\in(0,1), q=r/s, \ r,s\in\N, \]
and let $(g_{j})$ and $(h_{j})$ be $n$-tuple in $\SO(\R)^{n}$.
Assume that $\gmlhs{(g_{j})}{(h_{j})}{\La}$ is an idempotent element of $\textnormal{M}_{n}(\lmodule^{\R})$, i.e., $\{\pi(\la)g_{j}\}_{\la\in\La,j\in\Z_{n}}$ and $\{\pi(\la)h_{j}\}_{\la\in\La,j\in\Z_{n}}$ are dual Gabor frames for $L^{2}(\R)$. If $\alpha$ and $\beta$ are such that $\alpha\beta = a/M$ for some $a,M\in\N$ and 
\[ \tilde{\Lambda} = \begin{bmatrix} \alpha & q\alpha \\ 0 & \beta \end{bmatrix}\cdot\begin{bmatrix} \Z \\ \{0,1,\ldots,sM-1\} \end{bmatrix} \subset \frac{\alpha}{sa}\Z \times [0,\frac{sa}{\alpha}), \]
then the $n$-tuple 
$(\tilde{g}_{j})$ and $(\tilde{h}_{j})$ in $(\ell^{1}(\alpha (s a)^{-1}\Z))^{n}$ given by $\tilde{g}_{j} = c\, \restrict{\alpha (sa)^{-1}\Z}g_{j}$, $\tilde{h}_{j} = c\, \restrict{\alpha (sa)^{-1}\Z}h_{j}$, and where 
\[ c= \big(  \alpha (sa)^{-1} \big)^{1/2},\] 
are such that $\gmlhs{(\tilde{g}_{j})}{(\tilde{h}_{j})}{\widetilde{\La}}$ is an idempotent element of $\textnormal{M}_{n}(\lmodule^{\alpha (sa)^{-1}\Z})$, i.e., the two Gabor systems $\{\pi(\la) \tilde{g}_{j}\}_{\la\in\widetilde{\La},j\in\Z_{n}}$ and $\{\pi(\la) \tilde{h}_{j}\}_{\la\in\widetilde{\La},j\in\Z_{n}}$ are dual frames for $\ell^{2}(\alpha (sa)^{-1}\Z)$. Moreover, canonical pairs of dual Gabor frames are preserved by this process. 
\end{corollary}

\subsection{Sampling a certain Gabor frame for $L^{2}(\R\times\Z_{q})$ to a frame for $\ell^{2}(\Z\times\Z_{q})$}

Corollary \ref{cor:ex2} below follows by an application of Theorem \ref{th:sampl-non-sep} to projective modules constructed over noncommutative tori by Connes that these may be described in terms of Gabor frames was noted in \cite{dajalalu18}. Concretely, we have the following setup:
we take $G=\ghat=\R\times\Z_{q}$ (equipped with the Lebesgue measure on $\R$ and the counting measure on $\Z_{q}$), $(\omega,k)\in \R\times\Z_{q}$ acts as a character on $(x,m)\in \R$ by $(x,m)\mapsto e^{2\pi i (x\omega+mk/q)}$, $H=\alpha a^{-1}\Z\times\Z_{q}$ (equipped with the counting measure), $\widehat{H}=[0,a \, \alpha^{-1})\times\Z_{q}$ (equipped with its probability measure and addition modulo $(a \, \alpha^{-1},q)$), $(\omega,k)\in[0,a \, \alpha^{-1})\times\Z_{q}$ acts as a chatacter on $(x,m)\in \alpha a^{-1}\Z\times \Z_{q}$ by $(x,m)\mapsto e^{2\pi i (x\omega+mk/q)}$. The groups $\La,\tilde{\La},\Lac$ and $\tilde{\La}^{\circ}$ are equipped with the counting measure and given by
\begin{align*}
    & \Lambda = \begin{bmatrix} \alpha & 0 \\ r & 0 \\ 0 & \beta \\
0 & s \end{bmatrix}\cdot \begin{bmatrix} \Z \\ \Z \end{bmatrix},\\
& \tilde{\Lambda} = \begin{bmatrix} \alpha & 0 \\ r & 0 \\ 0 & \beta \\
0 & s \end{bmatrix}\cdot \begin{bmatrix} \Z \\ \{0,1,\ldots,qM-1\} \end{bmatrix} \subset \frac{\alpha}{a}\Z\times\Z_{q} \times [0,\frac{a}{\alpha})\times\Z_{q}, \\
& \Lac = \begin{bmatrix} (\beta q)^{-1} & 0 \\ -s^{\circ} & 0 \\ 0 & (\alpha q)^{-1} \\
0 & -r^{\circ}\end{bmatrix}\cdot \begin{bmatrix} \Z \\ \Z \end{bmatrix} \ \ \text{and} \ \ \tilde{\La}^{\circ} = \begin{bmatrix} (\beta q)^{-1} & 0 \\ -s^{\circ} & 0 \\ 0 & (\alpha q)^{-1} \\
0 & -r^{\circ}\end{bmatrix}\cdot \begin{bmatrix} \Z \\ \{0,1,\ldots, aq\} \end{bmatrix}, \end{align*}
where $r^{\circ}$ and $s^{\circ}$ are integers such that there exists $l_{r},l_{s}\in\Z$ such that
\[ r r^{\circ} + ql_{r}=1 \ \ \text{and} \ \ s s^{\circ} + ql_{s}=1. \]
Observe, $\textnormal{s}_{G\times\ghat}(\La) = \textnormal{s}_{H\times\widehat{H}}(\tilde{\La}) = q\alpha\beta = a/M$ and $\textnormal{s}_{G}(H) = \alpha a^{-1}$.
With the notation as in Remark \ref{rem:1901}, we take $K_{H^{\perp}} = [0,a \, \alpha^{-1})\times\{0\} \subseteq \R\times\Z_{q}$. Hence the map $\Phi: H\times\widehat{H}=\alpha a^{-1}\Z\times\Z_{q} \times [0,a \, \alpha^{-1})\times\Z_{q} \to G\times\ghat=\R\times\Z_{q}\times\R\times\Z_{q}$ is the one given by $\Phi(x,m,\omega,k) = (x,m,\omega,k)$. One can now verify that the conditions of Theorem \ref{th:sampl-non-sep} and Proposition \ref{pr:canonical-preserved}
are satisfied.
In this case the Banach algebras $\lmodule^{\R\times\Z_{q}}$ and $\rmodule^{\R\times\Z_{q}}$ are given by 
\begin{align*} 
\lmodule^{\R\times\Z_{q}} & = \big\{ {\bf a} \in \mathsf{B}(L^{2}(\R\times\Z_{q})) \, : \, {\bf{a}} = \sum_{\la\in\La} \mathsf{a}(\lambda) \, \pi(\lambda) , \ \mathsf{a}\in \ell^{1}(\Lambda)\big\}, \\
\rmodule^{\R\times\Z_{q}} & = \big\{ {\bf b} \in \mathsf{B}(L^{2}(\R\times\Z_{q})) \, : \, {\bf{b}} = \frac{1}{q\alpha\beta} \sum_{\lac\in\Lambda^{\circ}} \mathsf{b}(\lambda^{\circ}) \, \pi(\lambda^{\circ})^{*}, \ \mathsf{b}\in \ell^{1}(\Lambda^{\circ})\big\}, \end{align*}
and $\Vert {\bf a} \Vert_{\lmodule^{\R\times\Z_{q}}} = \Vert \mathsf{a} \Vert_{1}$ and $\Vert {\bf b} \Vert_{\lmodule^{\R\times\Z_{q}}} = \Vert \mathsf{b} \Vert_{1}$. Observe that $\lmodule^{\R\times\Z_{q}}$ is generated by the two unitary operators $U$ and $V$ on $L^{2}(\R\times\Z_{q})$ given by  
\[ Uf(t) = T_{\alpha,r} f(t,k) = f(t-\alpha,r-k) \ \ \text{and} \ \ Vf(t,k) = E_{\beta,s} f(t,k) = e^{2\pi i (\beta t+ks/q)} f(t,k), \ \ (t,k)\in\R\times\Z_{q},\]
where as $\rmodule^{\R\times\Z_{q}}$ is generated by the two unitaries ${U}^{\circ}$ and ${V}^{\circ}$ given by
\[ U^{\circ} f(t,k) = E_{(q\alpha)^{-1},-r^{\circ}} f(t,k) = e^{2\pi i (t (q\alpha)^{-1}-kr^{\circ}/q)} f(t,k)\]
and
\[ V^{\circ} f(t,k) = T_{(q\beta)^{-1},-s^{\circ}} f(t,k) = f(t-1/q\beta,k+s^{\circ}), \ \ (t,k)\in\R\times\Z_{q}.\]
The Banach algebras $\lmodule^{\alpha a^{-1}\Z \times\Z_{q}}$ and $\rmodule^{\alpha a^{-1}\Z \times\Z_{q}}$ are given by
\begin{align*} 
\lmodule^{\alpha a^{-1}\Z \times\Z_{q}} & = \big\{ {\bf a} \in \mathsf{B}(\ell^{2}(\alpha a^{-1}\Z \times\Z_{q})) \, : \, {\bf{a}} = \sum_{\la\in\tilde\La} \mathsf{a}(\lambda) \, \pi(\lambda) , \ \mathsf{a}\in \ell^{1}(\tilde\Lambda)\big\}, \\
\rmodule^{\alpha a^{-1}\Z \times\Z_{q}} & = \big\{ {\bf b} \in \mathsf{B}(\ell^{2}(\alpha a^{-1}\Z \times\Z_{q})) \, : \, {\bf{b}} = \frac{1}{q\alpha\beta} \sum_{\lac\in\tilde\Lambda^{\circ}} \mathsf{b}(\lambda^{\circ}) \, \pi(\lambda^{\circ})^{*}, \ \mathsf{b}\in \ell^{1}(\tilde\Lambda^{\circ})\big\}. \end{align*}
The algebra $\lmodule^{\alpha a^{-1}\Z \times\Z_{q}}$ is generated by the unitaries $\tilde{U}$ and $\tilde{V}$ on $\ell^{2}(\alpha a^{-1}\Z \times\Z_{q})$ given by
\begin{align*}
    & \tilde{U} f(t,k) = T_{\alpha,r}f(t,k) = f(t-\alpha,k-r), \\
    & \tilde{V}f(t,k) = E_{\beta,s} f(t,k) = e^{2\pi i (\beta t+s k/q)} f(t,k), \ \ (t,k)\in \alpha a^{-1}\Z \times\Z_{q},\end{align*}
and the algebra $\rmodule^{\alpha a^{-1}\Z\times\Z_{q}}$ is generated by the unitaries $\tilde{U}^{\circ}$ and $\tilde{V}^{\circ}$ on $\ell^{2}(\alpha a^{-1}\Z\times\Z_{q})$ given by
\[ \tilde{U}^{\circ} f(t,k) = E_{(q\alpha)^{-1},-r^{\circ}} f(t,k) = e^{2\pi i (t (q\alpha)^{-1}-kr^{\circ}/q)} f(t,k), \ (t,k)\in \alpha a^{-1}\Z\times\Z_{q},\]
and
\[ \tilde{V}^{\circ} f(t,k) = T_{(q\beta)^{-1},-s^{\circ}} f(t,k) = f(t-1/q\beta,k+s^{\circ}), \ \ (t,k)\in\R\times\Z_{q}.\]
Observe that 
\[ VU = e^{2\pi i (\alpha\beta+rs/q)}\, UV\ , \ \ \tilde{V} \tilde{U} = e^{2\pi i (\alpha\beta+rs/q)} \, \tilde{U} \tilde{V}.\]

\begin{corollary} \label{cor:ex2} Let $q,r,s\in\N$ and let $r$ and $s$ be such that they are co-prime to $q$. Let $\La$ be the lattice in $(\R\times\Z_{q})\times(\R\times\Z_{q})$ given by
\[ \Lambda = \begin{bmatrix} \alpha & 0 \\ r & 0 \\ 0 & \beta \\
0 & s \end{bmatrix}\cdot \begin{bmatrix} \Z \\ \Z \end{bmatrix}\]
and let $(g_{j})$ and $(h_{j})$ be $n$-tuple in $\SO(\R\times\Z_{q})^{n}$.
Assume that $\gmlhs{(g_{j})}{(h_{j})}{\La}$ is an idempotent element of $\textnormal{M}_{n}(\lmodule^{\R\times\Z_{q}})$, i.e., $\{\pi(\la)g_{j}\}_{\la\in\La,j\in\Z_{n}}$ and $\{\pi(\la)h_{j}\}_{\la\in\La,j\in\Z_{n}}$ are dual Gabor frames for $L^{2}(\R\times\Z_{q})$. If $\alpha,\beta$ and $q$ are such that $q\alpha\beta = a /M$ for some $a,M\in\N$ and we take
\[ \tilde{\Lambda} = \begin{bmatrix} \alpha & 0 \\ r & 0 \\ 0 & \beta \\
0 & s \end{bmatrix}\cdot \begin{bmatrix} \Z \\ \{0,1,\ldots,qM-1\} \end{bmatrix} \subset \frac{\alpha}{a}\Z\times\Z_{q} \times [0,\frac{a}{\alpha})\times\Z_{q}, \]
then the $n$-tuple 
$(\tilde{g}_{j})$ and $(\tilde{h}_{j})$ in $(\ell^{1}(\alpha a^{-1}\Z)\times\Z_{q}))^{n}$ given by $\tilde{g}_{j} = c\, \restrict{\alpha a^{-1}\Z\times\Z_{q}}g_{j}$, $\tilde{h}_{j} = c\, \restrict{\alpha a^{-1}\Z\times\Z_{q}}h_{j}$, and where 
\[ c= \big(  \alpha a^{-1} \big)^{1/2},\] 
are such that $\gmlhs{(\tilde{g}_{j})}{(\tilde{h}_{j})}{\widetilde{\La}}$ is an idempotent element of $\textnormal{M}_{n}(\lmodule^{\alpha a^{-1}\Z\times\Z_{q}})$, i.e., the two Gabor systems $\{\pi(\la) \tilde{g}_{j}\}_{\la\in\widetilde{\La},j\in\Z_{n}}$ and $\{\pi(\la) \tilde{h}_{j}\}_{\la\in\widetilde{\La},j\in\Z_{n}}$ are dual frames for $\ell^{2}(\alpha a^{-1}\Z\times\Z_{q})$. Moreover, canonical pairs of dual Gabor frames are preserved by this process. 
\end{corollary}

\subsection{Sampling and Periodizing Gabor frames from $L^{2}(\R)$ to a frame for $\C^{d}$.}

This example (for single window dual frames) can be found in \cite{so07}.
The result follows by an application of Theorem \ref{th:sampl-non-sep} followed by an application of Theorem \ref{th:period-non-sep} (or the other way around). We leave the details to the reader. 
The algebra $\lmodule^{\R}$ is generated by the unitary operators $U$ and $V$ on $L^{2}(\R)$ given by 
\[ Uf(t) = E_{\beta}f(t) = e^{2\pi i \beta t} f(t), \ Vf(t) = T_{\alpha}f(t) = f(t-\alpha), \ \ t\in\R\]
so that
\[ \lmodule^{\R} = \big\{ {\bf a} \in \mathsf{B}(L^{2}(\R)) \, : \, {\bf{a}} = \sum_{m,n\in\Z} \mathsf{a}(m,n) \, E_{m\beta}T_{n\alpha} , \ \mathsf{a}\in \ell^{1}(\Z^{2})\big\}.\]
The algebra $\lmodule^{\Z_{d}}$ is generated by the two unitary operators $\tilde{U}$ and $\tilde{V}$ that act on $\C^{d}$ given by
\[ \tilde{U}f(t) = E_{b}f(t) = e^{2\pi i b t/d}f(t) \ \ \tilde{V}f(t) = T_{a}f(t) = f(t-a), \ \ t\in\{0,1,\ldots,d-1\},\]
so that 
\[ \lmodule^{\Z_{d}} = \big\{ {\bf a} \in \mathsf{B}(\C^{d}) \, : \, {\bf{a}} = \sum_{m=0}^{M-1}\sum_{n=0}^{N-1}  \mathsf{a}(m,n) \, E_{mb}T_{na} , \ \mathsf{a}\in \C^{M\times N}\big\}\]
Observe that $UV = e^{2\pi i \alpha\beta} \, VU$ and $\tilde{U}\tilde{V} = e^{2\pi i \alpha\beta} \, \tilde{V}\tilde{U}$.

\begin{corollary} \label{cor:finite}
Let $\La$ be the lattice in $\R^{2}$ given by
\[ \Lambda = \begin{bmatrix} \alpha & 0 \\ 0 & \beta \end{bmatrix}\cdot\begin{bmatrix} \Z \\ \Z \end{bmatrix}, \ \ \alpha,\beta>0.\]
and let $(g_{j})$ and $(h_{j})$ be $n$-tuple in $\SO(\R)^{n}$.
Assume that $\gmlhs{(g_{j})}{(h_{j})}{\La}$ is an idempotent element of $\textnormal{M}_{n}(\lmodule^{\R})$, i.e., $\{\pi(\la)g_{j}\}_{\la\in\La,j\in\Z_{n}}$ and $\{\pi(\la)h_{j}\}_{\la\in\La,j\in\Z_{n}}$ are dual Gabor frames for $L^{2}(\R)$. If $\alpha$ and $\beta$ are such that $\alpha\beta = a/M=b/N$ for some $a,b,M,N\in\N$, we put $d=Mb(=aN)$, and let
\[ \tilde{\Lambda} = \begin{bmatrix} a & 0 \\ 0 & b \end{bmatrix}\cdot\begin{bmatrix} \{0,1,\ldots,N-1\} \\ \{0,1,\ldots,M-1\} \end{bmatrix} \subset \Z_{d} \times \Z_{d}, \]
then the $n$-tuple 
$(\tilde{g}_{j})$ and $(\tilde{h}_{j})$ in $(\C^{d})^{n}$ given by 
\[ \tilde{g}_{j}(t) = \sqrt{\alpha a^{-1}} \sum_{k\in\Z} g_{j}(\alpha a^{-1}(t-kd)), \ \ \tilde{h}_{j}(t) = \sqrt{\alpha a^{-1}} \sum_{k\in\Z} h_{j}(\alpha a^{-1}(t-kd)), \ \ t\in\{0,1,\ldots,d-1\}\] 
are such that $\gmlhs{(\tilde{g}_{j})}{(\tilde{h}_{j})}{\widetilde{\La}}$ is an idempotent element of $\textnormal{M}_{n}(\lmodule^{\Z_{d}})$, i.e., the two Gabor systems $\{\pi(\la) \tilde{g}_{j}\}_{\la\in\widetilde{\La},j\in\Z_{n}}$ and $\{\pi(\la) \tilde{h}_{j}\}_{\la\in\widetilde{\La},j\in\Z_{n}}$ are dual frames for $\C^{d}=\ell^{2}(\Z_{d})$. Moreover, canonical pairs of dual Gabor frames are preserved by this process. 
\end{corollary}

\subsection{Other examples}

In \cite{enjalu18} Gabor frames for $L^{2}(\R\times\mathbb{Q}_{p})$ and $L^{2}$ over the adeles were constructed. The sampling and periodization results from Theorem \ref{th:sampl-non-sep}
and \ref{th:period-non-sep} can be applied to this setting. Super (also called vector valued) Gabor frames \cite{ba00,grly09,jalu18} can also be sampled and periodized with these results. We leave these examples for elsewhere.

\section{Approximation of Heisenberg modules over irrational noncommutative tori}
\label{sec:from-discrete-to-continuous}

In this section we focus on the situation as described in the introduction, the rational non-commutative torus $\lmodule_{\theta}$, where $\theta$ is such that $\theta = a/M = b/N$ for some $a,b,M,N\in\N$ and we put $d= aN$. We use the results on generators of Heisenberg modules over rational noncommutative tori to approximate the generators of Heisenberg modules for irrational noncommutative tori. In particular, we show that representations of rational noncommutative tori on $\mathbb{C}^n$ yields approximations of generators in $S_0(\mathbb{R})$ by generators over finite-dimensional matrix algebras.

We consider its three different realizations as Banach algebras that act on $L^{2}(\R)$, $\ell^{2}(a^{-1}\Z)$ and $\ell^{2}(\Z/d\Z)\cong \C^{d}$, namely  $\lmodule_{\theta}^{\R}$, $\lmodule_{\theta}^{a^{-1}\Z}$ and $\lmodule_{\theta}^{\Z_{d}}$, respectively. In order to define them, we introduce the three lattices
\[ \Lambda = \Z\times\theta\Z \subset \R^{2}, \ \ \tilde{\Lambda} = \Z\times\theta \Z_{M} \subset a^{-1}\Z\times \R/a\Z, \ \ \tilde{\tilde{\Lambda}} = a\Z_{N} \times b\Z_{M}\subset (\Z/d\Z)^{2}.\]
We set
\begin{align*}
    \lmodule^{\R}_{\theta} & = \big\{ {\bf a} \in \mathsf{B}(L^{2}(\R)) \, : \, {\bf{a}} =  \sum_{\la\in \Lambda}  \mathsf{a}(\lambda) \, \pi(\lambda) \, , \ \mathsf{a}\in \ell^{1}(\Lambda)\big\}, \\
     \lmodule^{a^{-1}\Z}_{\theta} & = \big\{ {\bf a} \in \mathsf{B}(\ell^{2}(a^{-1}\Z)) \, : \, {\bf{a}} =  \sum_{\lambda \in\tilde{\La}} \mathsf{a}(\la) \, \pi(\la) \, , \ \mathsf{a}\in \ell^{1}(\tilde{\La}) \big\}, \\
     \lmodule^{\Z_d}_{\theta} & = \big\{ {\bf a} \in \mathsf{B}(\ell^{2}(\Z_{d})) \, : \, {\bf{a}} =   \sum_{\la\in \tilde{\tilde{\Lambda}}} \textsf{a}(\la) \, \pi(\la) \, , \ \mathsf{a}\in \ell^{1}(\La_{d})\big\},
\end{align*}
with the respective inner products
\begin{align*} \glhs{\,\cdot\,}{\,\cdot\,}{\La} & : \SO(\R)\times \SO(\R) \to \lmodule^{\R}_{\theta}, \ \glhs{f}{g}{\La} = \sum_{\la\in\La} \langle f, \pi(\la) g \rangle \, \pi(\la), \\
\glhs{\,\cdot\,}{\,\cdot\,}{\tilde{\La}} & : \ell^{1}(a^{-1} \Z)\times \ell^{1}(a^{-1}\Z) \to \lmodule^{a^{-1}\Z}_{\theta}, \ \glhs{f}{g}{\tilde{\La}} = \sum_{\la\in\tilde{\La}} \langle f, \pi(\la) g \rangle \, \pi(\la) ,\\
\glhs{\,\cdot\,}{\,\cdot\,}{\tilde{\tilde{\Lambda}}} & : \C^{d}\times \C^{d} \to \lmodule^{\Z_d}_{\theta}, \ \glhs{f}{g}{\tilde{\tilde{\Lambda}}} = \sum_{\la\in \tilde{\tilde{\Lambda}}} \langle f, \pi(\la) g \rangle \, \pi(\la).
\end{align*}
In a similar fashion we realize the non-commutative torus with parameter $\theta^{-1}$ as operators on either of the three spaces $L^{2}(\R)$, $\ell^{2}(a^{-1}\Z)$, and $\ell^{2}(\Z/d\Z)$. We define
\[ \Lac = \theta^{-1}\Z\times\Z \subset \R^{2} \ \ \tilde{\La}^{\circ} = \theta^{-1}\Z\times\Z_{M} \subset a^{-1}\Z\times \R/a\Z, \ \ \tilde{\tilde{\Lambda}}^{\circ} = M\Z_{b} \times N\Z_{a}\subset (\Z/d\Z)^{2}.\]
With those lattices we construct
\begin{align*}
    & \rmodule^{\R}_{1/\theta} = \big\{ {\bf b} \in \mathsf{B}(L^{2}(\R)) \, : \, {\bf{b}} =  \theta^{-1}\sum_{\lac\in\Lac}  \mathsf{b}(\lac) \, \pi(\lac)^{*} \, , \ \mathsf{b}\in \ell^{1}(\Lac)\big\}, \\
    & \rmodule^{a^{-1}\Z}_{1/\theta} = \big\{ {\bf b} \in \mathsf{B}(\ell^{2}(a^{-1}\Z)) \, : \, {\bf{b}} =  \theta^{-1}\sum_{\lac \in\tilde{\La}^{\circ}}  \mathsf{b}(\lac) \, \pi(\lac)^{*} \, , \ \mathsf{b}\in \ell^{1}(\tilde{\La}^{\circ}) \big\}, \\
    & \rmodule^{\Z_d}_{1/\theta} = \big\{ {\bf b} \in \mathsf{B}(\ell^{2}(\Z_{d})) \, : \, {\bf{b}} = \theta^{-1}  \sum_{\lac\in \tilde{\tilde{\Lambda}}^{\circ}} \textsf{b}(\lac) \, \pi(\lac)^{*} \, , \ \mathsf{b}\in \ell^{1}(\Lac_{d})\big\},
\end{align*}
with the respective inner products
\begin{align*} \grhs{\,\cdot\,}{\,\cdot\,}{\Lac} & : \SO(\R)\times \SO(\R) \to \rmodule^{\R}_{1/\theta}, \ \grhs{f}{g}{\R} = \theta^{-1} \sum_{\lac\in\Lac} \langle g, \pi(\lac)^{*} f \rangle \, \pi(\lac)^{*}, \\
\grhs{\,\cdot\,}{\,\cdot\,}{\tilde{\La}^{\circ}} & : \ell^{1}(a^{-1} \Z)\times \ell^{1}(a^{-1}\Z) \to \rmodule^{a^{-1}\Z}_{1/\theta}, \ \grhs{f}{g}{a^{-1}\Z} = \theta^{-1} \sum_{\lac\in\tilde{\La}^{\circ}} \langle g, \pi(\lac)^{*} f \rangle \, \pi(\lac)^{*} ,\\
\grhs{\,\cdot\,}{\,\cdot\,}{\tilde{\tilde{\Lambda}}^{\circ}} & : \C^{d}\times \C^{d} \to \rmodule^{\Z_d}_{1/\theta}, \ \grhs{f}{g}{\tilde{\tilde{\Lambda}}^{\circ}} = \theta^{-1} \sum_{\lac\in\tilde{\tilde{\Lambda}}^{\circ}} \langle g, \pi(\lac)^{*} f \rangle \, \pi(\lac)^{*}.
\end{align*}

These inner product are all such that
\begin{align*}
     \glhs{f}{g}{\La} \cdot h & = f \cdot \grhs{g}{h}{\Lac} \ \ \text{for all} \ \ f,g,h\in\SO(\R), \\
     \glhs{f}{g}{\tilde{\La}} \cdot h & = f \cdot \grhs{g}{h}{\tilde{\La}^{\circ}} \ \ \text{for all} \ \ f,g,h\in\ell^{1}(a^{-1}\Z), \\
     \glhs{f}{g}{\tilde{\tilde{\Lambda}}}\cdot h & = f \cdot \grhs{g}{h}{\tilde{\tilde{\Lambda}}^{\circ}} \ \ \text{for all} \ \ f,g,h\in\C^{d}\cong\ell^{2}(\Z_{d}).
\end{align*}

We equip the spaces of $\lmodule_{\theta}$- and $\rmodule_{1/\theta}$-valued $n\times n$ matrices with the inner products as defined in Section \ref{sec:frames-and-modules}. 

In the following, for $\gamma\in\R$ we let $\restrict{\gamma\Z}$ be the restriction operator that takes a function in $\SO(\R)$ to a sequence in $\ell^{1}(\gamma\Z)$ defined by $\big(\restrict{\gamma\Z}f\big)(k) = f(k)$, $k\in\gamma\Z$. 
Furthermore, for $\gamma\in\R\backslash\{0\}$, $d\in\N$  we let $\period{\gamma d\Z}$ be the operator
\[ \period{\gamma d\Z} : \ell^{1}(\gamma \Z) \to \C^{d}, \ \big(\period{\gamma d\Z}f\big)(t) = \sum_{k\in\gamma d\Z} f(\gamma t-k), \ \ t\in \{0,1,\ldots,d-1\}.\]

We now generalize Theorem \ref{th:intro} to the multi-generator case.
\begin{theorem} \label{th:intro-multi}
Let all notation be as above. If $(g_{j})$ is an $n$-tuple of functions in $\SO(\R)^{n}$ such that $\gmlhs{(g_{j})}{(g_{j})}{\La}$ is a projection in $\textnormal{M}_{n}(\lmodule^{\R}_{\theta})$, then the following holds.
\begin{enumerate}
    \item[(i)] The module norm of $(g_{j})$ satisfies
    \[ \Vert (g_{j}) \Vert_{\La} \le C := \theta^{-1} \sum_{j\in\Z_{n}} \sum_{\lac\in\theta^{-1}\Z\times\Z} \vert \langle g_{j}, \pi(\lac) g_{j}\rangle \vert.\]
    \item[(ii)] The $n$-tuple in $\ell^{1}(a^{-1}\Z)^{n}$ defined by $(\tilde{g}_{j})= ( a^{-1/2} \, \restrict{a^{-1} \Z} g_{j})$ is such that $\gmlhs{(\tilde{g}_{j})}{(\tilde{g}_{j})}{\tilde{\La}}$ is a projection in $\textnormal{M}_{n}(\lmodule^{a^{-1}\Z}_{\theta})$. Moreover, the module norm of $(\tilde{g}_{j})$ satisfies $\Vert (\tilde{g}_{j}) \Vert_{\tilde{\La}} \le C$.
    \item[(iii)] The $n$-tuple of vectors in $\C^{d}$ given by $(\tilde{\tilde{g}}_{j}) = (a^{-1/2} \period{a^{-1} d\Z} \restrict{a^{-1} \Z} g_{j})$ is such that $\gmlhs{(\tilde{\tilde{g}}_{j})}{(\tilde{\tilde{g}}_{j})}{\tilde{\tilde{\La}}}$ is a projection in $\textnormal{M}_{d}(\lmodule^{\Z_{d}}_{\theta})$. Moreover, the module norm of $(\tilde{\tilde{g}}_{j})$ satisfies $\Vert (\tilde{\tilde{g}}_{j}) \Vert_{\tilde{\tilde{\La}}} \le C$. 
\end{enumerate}
\end{theorem}
\begin{proof}
This follows from Theorem \ref{th:period-non-sep} and Theorem \ref{th:sampl-non-sep}. See Corollary \ref{cor:finite}.
\end{proof}

We now go a step further, and translate further results from the theory of Gabor frames and time-frequency analysis into the setting of the non-commutative tori considered here. This will show that it is also possible to go back from the discrete world into to the continuous one.
The hard work in order to establish the following result has already been done in  \cite{feka04, feka07, ka05}.

The way we will construct a function in $\SO(\R)$ from a sequence in $\ell^{1}$ is by use of linear interpolation.
For a given $\gamma>0$ let $\wedge_{\gamma}$ be the triangular-function
\[ \wedge_{\gamma}(x) = \big( 1 - \vert \gamma^{-1} x \vert  \big) \cdot \mathds{1}_{[-\gamma,\gamma]}. \]
Furthermore, for any $\gamma>0$ we define the operator
\[ Q^{\R}_{\gamma\Z} : \ell^{1}(\gamma\Z) \to \SO(\R), \ \big(Q^{\R}_{\gamma} c\big)(t) = \sum_{k\in\gamma\Z} c(k) \cdot \wedge_{\gamma}(t-k).\]
It is straight-forward to show that $Q^{\R}_{\gamma\Z}$ is well-defined, linear and bounded with $\Vert Q^{\R}_{\gamma\Z} \Vert_{\textnormal{op}} = \Vert \wedge_{\gamma} \Vert_{\SO}$. 
Observe that $Q^{\R}_{\gamma\Z}$ interpolates linearly in between the points $(k, c(k))_{k\in\gamma\Z}$. 

The procedure to turn a vector in $\C^{d}$ into a sequence in $\ell^{1}(\gamma \Z)$ is similar:
\[ Q^{\gamma\Z}_{d} : \C^{d} \to \ell^{1}(\gamma\Z), \ \big( Q^{\gamma\Z}_{d} c \big)(k) = \begin{cases} 0 & \text{if} \ k \notin\gamma\{-\lfloor \tfrac{d-1}{2} \rfloor, \ldots, \lfloor \tfrac{d}{2} \rfloor\} , \\ 
c(\gamma^{-1}k\ \text{mod} \ d)& \text{if} \ k \in \gamma \{-\lfloor \tfrac{d-1}{2} \rfloor, \ldots, \lfloor \tfrac{d}{2} \rfloor\}, \end{cases} \ \ k\in\gamma\Z.\]
Here $c\in\C^{d}$ is treated as a function that can be evaluated on the set $\{0,1,\ldots,d-1\}$.

We have the following important approximation results.
\begin{lemma}  \label{le:SO-R-approx-by-Z} Let all notation be as above. The following holds.
\begin{enumerate}
    \item[(i)] For any $f\in \SO(\R)$
\[ \lim_{\gamma\to 0} \big\Vert f - Q^{\R}_{\gamma\Z} \restrict{\gamma\Z} f\big\Vert_{\SO(\R)} = 0.\]
\item[(ii)] For any $f\in \SO(\R)$
\[ \lim_{n\to\infty} \big\Vert f - Q_{n^{-1/2}\Z}^{\R} Q^{n^{-1/2}\Z}_{n} \period{n^{1/2}\Z} \restrict{n^{-1/2}\Z} f \big\Vert_{\SO(\R)} = 0.\]
\end{enumerate}
\end{lemma}
\begin{proof}
(i) is \cite[Theorem 2.2]{feka07}. (ii) is \cite[Proposition 3]{ka05}.
\end{proof}

Theorem \ref{th:intro-multi} shows that the samples of generator of a matrix valued projection in $\textnormal{M}_{n}(\lmodule^{\R}_{\theta})$ also generate a projection in $\textnormal{M}_{n}(\lmodule^{a^{-1}\Z}_{\theta})$. If these samples are dense enough, and linear interpolated to a function on $\R$ by the operator $Q^{\R}_{\gamma\Z}$, then the in this way constructed collections of functions generate a projection in  $\textnormal{M}_{n}(\lmodule^{\R}_{\theta})$ again.
\begin{theorem} \label{th:from-discrete-to-continuous} Let all notation and assumptions be as above. Let $(g_{j})$ be an $n$-tuple in $\SO(\R)^{n}$ such that $\gmlhs{(g_{j})}{(g_{j})}{\La}$ is a projection in $\textnormal{M}_{n}(\lmodule^{\R}_{\theta})$.
The following holds.
\begin{enumerate}
    \item[(i)] For all $a\in\N$ let $(\tilde{g}_{j})$ be the $n$-tuple in $\ell^{1}(a^{-1}\Z)^{n}$ given by 
    \[ (\tilde{g}_{j}) = ( \restrict{a^{-1} \Z} g_{j}). \] 
    For all $a$ that are sufficiently large, the $n$-tuple $(k_{j}) := (Q^{\R}_{a^{-1}\Z} \tilde{g}_{j})$ in $\SO(\R)^{n}$ is such that $\mathbf{b}_{(k_{j})} := \gmrhs{(k_{j})}{(k_{j})}{\Lac}\in \textnormal{M}_{n}(\rmodule^{\R}_{\theta^{-1}})$ is invertible on $L^{2}(\R)^{n}$ and on $\SO(\R)^{n}$. In that case, the $n$-tuple $(k_{j})\cdot \mathbf{b}_{(k_{j})}^{-1}\in \SO(\R)^{n}$ is such that $\gmlhs{(k_{j})}{(k_{j})\cdot \mathbf{b}_{(k_{j})}^{-1}}{\La}$ is a  projection in $\textnormal{M}_{n}(\lmodule^{\R}_{\theta})$.
    Furthermore, we also have that $(k_{j}) = (Q^{\R}_{a^{-1}\Z} \tilde{g}_{j})$ converges towards $(g_{j})$ in the module norm as $a\to\infty$, that is 
    \[\lim_{a\to\infty} \Vert (g_{j}) - (k_{j}) \Vert_{\La} = 0.\]
    \item[(ii)] For each $d\in\N$ let $(\tilde{\tilde{g}}_{j})$ be the $n$-tuple of vectors in $\C^{d}$ given by
\[ (\tilde{\tilde{g}}_{j}) = (\period{d^{1/2}\Z} \restrict{d^{-1/2} \Z} g_{j}).\]
For all $d$ that are sufficiently large, the $n$-tuple of functions $(k_{j})$ in $\SO(\R)^{n}$ given by
\[ (k_{j}):= ( Q_{d^{-1/2}\Z}^{\R} Q_{d}^{d^{-1/2}\Z} \tilde{\tilde{g}}_{j})\]
is such that $\mathbf{b}_{(k_{j})} := \gmrhs{(k_{j})}{(k_{j})}{\Lac}\in \textnormal{M}_{n}(\rmodule^{\R}_{\theta^{-1}})$ is invertible on $L^{2}(\R)^{n}$ and on $\SO(\R)^{n}$. In that case, the $n$-tuple $(k_{j})\cdot \mathbf{b}_{(k_{j})}^{-1}\in \SO(\R)^{n}$ is such that $\gmlhs{(k_{j})}{(k_{j})\cdot \mathbf{b}_{(k_{j})}^{-1}}{\La}$ is a  projection in $\textnormal{M}_{n}(\lmodule^{\R}_{\theta})$.
    Furthermore, we also have that $(k_{j})$ converges towards $(g_{j})$ in the module norm as $d\to\infty$, that is 
    \[\lim_{d\to\infty} \Vert (g_{j}) - (k_{j}) \Vert_{\La} = 0.\]

\end{enumerate}
\end{theorem}
\begin{proof}
(i). By Lemma \ref{le:SO-R-approx-by-Z} it follows that for each $j=\{0,1,\ldots,n-1\}$ the limit 
\[\lim_{a\to \infty} \Vert g_{j}-Q^{\R}_{a^{-1}\Z}\restrict{a^{-1}\Z}g_{j}\Vert_{\SO(\R)} = 0.\] The main result by Feichtinger and Kaiblinger in \cite{feka04}, states that for all $a$ sufficiently large the frame property of the Gabor system generated by $(g_{j})$ and the time-frequency shifts from $\La=\Z\times\theta \Z$ implies that also the Gabor system generated by the defined $n$-tuple $(k_{j})$ and $\La$ is a frame for $L^{2}(\R)$. This is equivalent to the fact that $\mathbf{b}_{(k)_{j}}$ is invertible as an operator on $L^{2}(\R)^{n}$ and $\SO(\R)^{n}$. That $\gmlhs{(k_{j})}{(k_{j})\cdot \mathbf{b}_{(k_{j})}^{-1}}{\La}$ is a projection is easily verified with the properties of the $\textnormal{M}_{n}(\lmodule_\theta^{\R})$-valued inner product (use the compatibility with the left and right action from $\textnormal{M}_{n}(\lmodule_\theta^{\R})$ and $\textnormal{M}_{n}(\rmodule_{1/\theta}^{\R})$ and the fact that $\mathbf{b}_{(k_{j})}$ is self-adjoint). 
Finally, we have the estimates

\begin{align*}
    & \Vert (g_{j}) - (k_{j}) \Vert_{\La}^{2} = \Vert (g_{j}) - (k_{j}) \Vert_{\Lac}^{2} = \big\Vert \gmrhs{(g_{j}-k_{j})}{(g_{j}-k_{j})}{\Lac} \big\Vert_{\textnormal{op},L^{2}} \\
    & \le \theta^{-1} \sum_{j\in\Z_{n}}\sum_{\lac\in\Lac} \big\vert \big\langle g_{j}-k_{j} \, , \, \pi(\lac)^{*} g_{j}- \pi(\lac)^{*} k_{j} \big\rangle_{L^{2}(\R)} \big\vert \\ 
    & = \theta^{-1} \sum_{j\in\Z_{n}}\sum_{\lac\in\Lac} \big\vert \mathcal{V}_{g_{j}-k_{j}} (g_{j}-k_{j})(\lac)\big\vert \\
    & \le c \, \sum_{j\in\Z_{n}} \Vert g_{j}-k_{j}\Vert_{\SO}^{2}  = c \, \sum_{j\in\Z_{n}} \Vert g_{j}-Q^{\R}_{a^{-1}\Z} \restrict{a^{-1}\Z} {g}_{j}\Vert_{\SO}^{2},
\end{align*}
for some $c>0$ that does not depend on $a$ or $(g_{j})$. The first inequality follows from Lemma \ref{le:so-implies-bessel}. The second inequality follows from Lemma \ref{le:s0-properties}(vi)+(viii). 
All that is left is to refer to Lemma \ref{le:SO-R-approx-by-Z} to conclude that the last term is indeed converging to zero as $a$ becomes larger. We conclude that $\lim_{a\to\infty} \Vert (g_{j}) - (k_{j}) \Vert_{\La} = 0$ as desired.

(ii). The reasoning for this statement is similar, we just use the second statement of Lemma \ref{le:SO-R-approx-by-Z}, rather than the first. 
\end{proof}

The theory of quantum Gromov-Hausdorff distance \cite{ri04-1} provides a notion of ``closeness'' between two $C^*$-algebras and thus provides a way to formalize the convergence of a sequence of $C^*$-algebras $(A_n)_{n\in\mathbb{N}}$  to a $C^*$-algebra $B$. In many problems one wants to approximate $B$ by finite-dimensional $C^*$-algebras $(A_n)_{n\in\mathbb{N}}$ like in the case of noncommutative tori. In addition, one is often interested in the behaviour of finitely generated projective modules over $A_n$ and if it in the ``limit'' it converges to a finitely generated projective module over $B$. The definition of a distance between modules over two $C^*$-algebras has been recently addressed in \cite{la16-3,ri18-1} and Latremoliere has worked out the case of Heisenberg modules over noncommutative 2-tori in \cite{la17-1,la18}. Here we complement these results by pointing out that if one uses that Heisenberg modules are finitely generated and projective over noncommutative tori, then one can study the behaviour of sequences of Heisenberg modules converging to a Heisenberg modules by understanding what is going on for the generators. 

Finally, we mention the following. Observe that Theorem \ref{th:intro-multi} only applies to rational $\theta$. If we translate results from \cite{feka04} and \cite{grle04} from statements of Gabor analysis into statements of the non-commutative torus, then we find the following.
\begin{theorem} \label{th:intro-irr}
Let all notation be as above. Furthermore, for a given $\tilde{\theta}$ 
we put $\La_{\tilde{\theta}} = \Z\times\tilde{\theta}\Z$ and $\La_{\tilde{\theta}}^{\circ} = \tilde\theta^{-1}\Z\times\Z$. If $(g_{j})$ is an $n$-tuple in $\SO(\R)^{n}$ such that $\gmlhs{(g_{j})}{(g_{j})}{\La}$, $\La=\Z\times\theta\Z$, is a projection in $\textnormal{M}_{n}(\lmodule^{\R}_{\theta})$ for some irrational $\theta$, then, for every rational $\tilde{\theta}$ that is sufficiently close to $\theta$, the element $\mathbf{b}_{(g_{j})} := \gmrhs{(g_{j})}{(g_{j})}{\Lac_{\tilde{\theta}}}\in \textnormal{M}_{n}(\rmodule^{\R}_{\tilde{\theta}})$ is invertible as an operator on $L^{2}(\R)^{n}$ and on $\SO(\R)^{n}$. Moreover, the element $(h_{j}):=(g_{j}) \cdot \mathbf{b}_{(g_{j})}
    ^{-1}$ in $\SO(\R)^{n}$ is such that  $\gmlhs{(g_{j})}{(h_{j})}{\La_{\tilde{\theta}}}$ is a projection in $\textnormal{M}_{n}(\lmodule^{\R}_{\tilde{\theta}})$.
 \end{theorem}
\begin{proof}
It follows from the main result of \cite{feka07} that every rational $\tilde{\theta}$ that is close enough to $\theta$ is such that the functions $(g_{j})$ generate a multi-window Gabor frame. Equivalently, the element $\mathbf{b}_{(g_{j})}\in \textnormal{M}_{n}(\rmodule^{\R}_{\tilde{\theta}})$ is invertible on $L^{2}(\R)^{n}$ and on $\SO(\R)^{n}$. Similar as in the proof of Theorem \ref{th:intro-multi} we show that $\gmlhs{(g_{j})}{(g_{j})\cdot \mathbf{b}_{(g_{j})}^{-1}}{\La_{\tilde{\theta}}}$ is a projection.
\end{proof}

Using Theorem \ref{th:intro-irr} it is therefore possible to go from the irrational to the rational case and subsequently apply Theorem \ref{th:intro-multi} and or Theorem \ref{th:from-discrete-to-continuous}. 

\newcommand{\ind}{i} 

\begin{proposition} \label{pr:2711}
 Let $\theta$ be irrational and take $(g_{j})$ to be an $n$-tuple in $\SO(\R)^{n}$ such that $\gmlhs{(g_{j})}{(g_{j})}{\La}$, $\La = \Z\times\theta\Z$, is a projection in $\textnormal{M}_{n}(\lmodule^{\R}_{\theta})$. If $(\theta_{\ind})_{\ind\in\N}$ is a sequence of rational numbers such that $\lim_{\ind\to\infty} \vert \theta - \theta_{\ind}\vert = 0$, $\La_{\ind}=\Z\times\theta_{\ind}\Z$, and $(a_{\ind})$, $(b_{\ind})$, $(N_{\ind})$ and $(M_{\ind})$ are sequences of natural numbers such that $\theta_{\ind} = a_{\ind}/M_{\ind} = b_{\ind}/N_{\ind}$ for all $\ind\in\N$ and such that the sequence $(d_{\ind})=(a_{\ind}\cdot N_{\ind})$ is increasing and unbounded, then the following holds.
 \begin{enumerate}
     \item[(i)] For all $\ind$ that are sufficiently large the $n$-tuple of vectors in $\C^{d_{\ind}}$ given by 
     \begin{align*} (\tilde{\tilde{g}}_{j}) & = (a^{-1/2}_{\ind} \period{a^{-1}_{\ind} d_{\ind}\Z} \restrict{a^{-1}_{\ind} \Z} g_{j}) \\
     \tilde{\tilde{g}}_{j}(t) & = a_{\ind}^{-1/2} \sum_{k\in\Z} g_{j}(a_{\ind}t+a_{\ind}d_{\ind} k), \ \ t\in\{0,1,\ldots,d_{\ind}\}, \ j\in\{0,1,\ldots,n-1\}, \end{align*}
     is such that $\gmlhs{(\tilde{\tilde{g}}_{j})}{(\tilde{\tilde{g}}_{j})}{\tilde{\tilde{\La}}}$, $\tilde{\tilde{\La}}=a_{\ind} \Z_{N_{\ind}} \times b_{\ind} \Z_{M_{\ind}}\subset \Z_{d_{\ind}}^{2}$ is a projection in $\textnormal{M}_{d_{\ind}}(\lmodule^{\Z_{d_{\ind}}}_{\theta_{\ind}})$.
     \item[(ii)] For all $\ind$ that are sufficiently large the $n$-tuple of functions in $\SO(\R)^{n}$ given by
     \begin{align*}
         & (k_{j}^{(\ind)}) = ( a_{\ind}^{1/2} Q_{d_{\ind}^{-1/2}\Z}^{\R} Q_{d_{\ind}}^{d_{\ind}^{-1/2}\Z} \tilde{\tilde{g}}_{j})
     \end{align*} 
     
is such that $\mathbf{b}_{(k_{j}^{(\ind)})} := \gmrhs{(k_{j}^{(\ind)})}{(k_{j}^{(\ind)})}{\Lac_{\ind}}\in \textnormal{M}_{n}(\rmodule^{\R}_{\theta_{\ind}^{-1}})$ is invertible on $L^{2}(\R)^{n}$ and on $\SO(\R)^{n}$. In that case, the $n$-tuple $(k_{j}^{(\ind)})\cdot \mathbf{b}_{(k_{j}^{(\ind)})}^{-1}\in \SO(\R)^{n}$ is such that $\gmlhs{(k_{j}^{(\ind)})}{(k_{j}^{(\ind)})\cdot \mathbf{b}_{(k_{j}^{(\ind)})}^{-1}}{\La_{\ind}}$ is a  projection in $\textnormal{M}_{n}(\lmodule^{\R}_{\theta_{\ind}})$. 
     Furthermore, we also have that $(k_{j}^{(\ind)})$ converges towards $(g_{j})$ in the module norm as $\ind\to\infty$, that is     \[\lim_{\ind\to\infty} \Vert (g_{j}) - (k_{j}^{(\ind)}) \Vert_{\La} = 0.\]
 \end{enumerate}
\end{proposition}
\begin{proof}
(i). This follows by Theorem \ref{th:intro-irr} together with Theorem \ref{th:from-discrete-to-continuous}(ii). \\
(ii). Since $(d_{\ind})$ is increasing and unbounded, it follows from Lemma \ref{le:SO-R-approx-by-Z} that for each $j$ the function $k_{j}^{(i)}$ converges towards $g_{j}$ in the $\SO$-norm. Hence, as in the proof of Theorem \ref{th:from-discrete-to-continuous} and \ref{th:intro-irr}, we conclude together with the main result of \cite{feka07} that the $n$-tuple $(k_{j}^{(i)})$ has all the desired properties with the convergence in the module norm.
\end{proof}

Let us note that the sequence of rational noncommtuative tori $(A_{\theta_{\ind}})$ converges in the quantum Gromov-Hausdorff distance to $A_\theta$, see \cite{ri04-1}. Hence Proposition \ref{pr:2711} shows that Heisenberg modules over $(A_{\theta_{\ind}})$  are also close to Heisenberg modules over $A_\theta$ in the Heisenberg module norm. 

\subsection*{Acknowledgments}
The work of M.S.J.\ was carried out during the tenure of the ERCIM ``Alain Bensoussan'' Fellowship Programme at NTNU. This project was completed while both authors were visiting the Faculty of Mathematics at the University of Vienna and the second author also acknowledges the hospitality of the Erwin Schr\"odinger Insitute while attending the program on ``Bi-Variant K-theory and its applications to physics''. The first author thanks Jordy T. van Velthoven for discussions on the embeddings in Remark \ref{rem:1901} and \ref{rem:1902}.

\bibliographystyle{abbrv}

\end{document}